\documentclass[11pt]{article}

\usepackage{mathrsfs}
\usepackage{amsmath,color}
\usepackage{amsthm}
\usepackage{amssymb}
\usepackage{amsfonts}
\usepackage{epsf}
\usepackage{graphicx}
\usepackage{hyperref}
\usepackage{esint}

\usepackage{overpic}
\usepackage[normalem]{ulem}

\newtheorem{thm}{Theorem}[section]
\newtheorem{cor}[thm]{Corollary}
\newtheorem{lem}[thm]{Lemma}
\newtheorem{prop}[thm]{Proposition}

\newtheorem{conj}[thm]{Conjecture}

\theoremstyle{definition}

\theoremstyle{remark}
\newtheorem{rem}{Remark}[section]

 \DeclareMathOperator\Ai{{Ai}}
 \DeclareMathOperator\re{{Re}}

\numberwithin{equation}{section}

\begin{document}

\title{Location of Poles for the Hastings-McLeod Solution
to the Second Painlev\'{e} Equation}
\gdef\shorttitle{Pole Location for the Hastings-McLeod Solution}

\author{Min Huang\footnotemark[1], ~Shuai-Xia Xu\footnotemark[2] ~and Lun Zhang\footnotemark[3]}

\date{}

\maketitle
\renewcommand{\thefootnote}{\fnsymbol{footnote}}
\footnotetext[1]{Department of Mathematics, City University of Hong Kong, Tat Chee
Avenue, Kowloon, Hong Kong. E-mail: mihuang\symbol{'100}cityu.edu.hk}
\footnotetext[2]{Institut Franco-Chinois de l'Energie Nucl\'{e}aire, Sun Yat-sen University,
Guangzhou 510275, China. E-mail: xushx3\symbol{'100}mail.sysu.edu.cn}
\footnotetext[3] {School of Mathematical Sciences and Shanghai Key Laboratory for Contemporary Applied Mathematics, Fudan University, Shanghai 200433, China. E-mail: lunzhang\symbol{'100}fudan.edu.cn }
\begin{abstract}
We show that the well-known Hastings-McLeod solution to the second Painlev\'{e} equation is pole-free in the region $\arg x \in [-\frac{\pi}{3},\frac{\pi}{3}]\cup [\frac{2\pi}{3},\frac{ 4 \pi}{3}]$, which proves an important special case of a general conjecture concerning pole distributions of Painlev\'{e} transcedents proposed by Novokshenov. Our strategy is to construct explicit quasi-solutions approximating the Hastings-McLeod solution in different regions of the complex plane, and estimate the errors rigorously. The main idea is very similar to the one used to prove Dubrovin's conjecture for the first Painlev\'{e} equation, but there are various technical improvements.



\end{abstract}

\section{Introduction}
Painlev\'{e} equations are six second order nonlinear ordinary differential equations first studied by Painlev\'{e} and his colleagues around 1900. They are well known for the so-called Painlev\'{e} property, i.e., the only movable singularities of their solutions are (finite order) poles; see \cite[\S 32.2]{DLMF}. Here `movable' means the location of the singularities (which in general can be poles, essential singularities or branch points) of the solutions depend on the constants of integration associated with the initial or boundary conditions of the differential equations. The solutions of these equations, often called the Painlev\'{e} transcendents \cite{DLMF}, in general cannot be represented in terms of elementary functions or known classical special functions. They play important roles in both pure and applied mathematics, and are widely thought of as the nonlinear counterparts of the classical special functions.

For the first two Painlev\'{e} equations
\begin{align}
\textrm{PI} \qquad &y''=6y^2+x,  \nonumber \\
\textrm{PII} \qquad &y''=2y^{3}+xy+\alpha, \label{eq:p2a}
\end{align}
all solutions are meromorphic in the complex plane with $x=\infty$ being the only essential singularity. The locations of the movable poles for the Painlev\'{e} transcendents are crucial for understanding a number of problems arising from mathematical physics; cf. \cite{BT,DGK,Masoero-a,Masoero-b}. In the pioneering works \cite{Boutroux-a,Boutroux-b}, Boutroux established the ``deformed'' elliptic function approximations in appropriate sectors near infinity, which leads to the degeneration of lattices of poles along the critical rays
\begin{equation}\label{def:gammak}
\Gamma_k:=\left\{x ~\Big{|}~ \arg x=\frac{2 k \pi  }{N} \right\}, \qquad k=0,1,\ldots, N-1,
\end{equation}
where
\begin{equation*}
N=\left\{
    \begin{array}{ll}
      5, & \hbox{for PI,} \\
      6, & \hbox{for PII.}
    \end{array}
  \right.
\end{equation*}
This means that the poles tend to align themselves along certain smooth curves which tend to one of the rays $\Gamma_k$ near infinity. Furthermore, Boutroux also showed the existence of solutions which have no lines of poles
near infinity near $n$ ($n=1,2,3$) of the critical rays $\Gamma_k$, which are called $n$-truncated solutions.

An interesting feature of the $2$- or $3$-truncated solutions is, as confirmed by numerical studies in \cite{FW1,FW2,Novo09}, that the distributions of poles near infinity characterize the global behavior of the poles.
More precisely, let $\Xi_k$ be the sector bounded by two consecutive critical rays:
\begin{equation*}
\Xi_k:=\left\{x ~\Big{|}~ \frac{2 k \pi }{N}<\arg x<\frac{2 (k+1) \pi }{N} \right\}, \qquad k=0,1,\ldots, N-1.
\end{equation*}
The following conjecture was made in \cite{Novo14} by Novokshenov:
\begin{conj}\label{conj}
If the $2$- or $3$-truncated solution of Painlev\'{e} equation has no pole at infinity in a sector $\Xi_k$, then it has no poles in the whole sector $\Xi_k$.
\end{conj}
For the $3$-truncated solutions of PI, a special case of this conjecture is known as Dubrovin's conjecture, which appeared in \cite{DGK} with connections to the critical behavior of the nonlinear Schr\"{o}dinger equation. It was proved recently in \cite{dub} with a technique developed in \cite{lpm}; see also \cite{Kap,Masoero-a,Masoero-b} for partial results.

In this paper, we will further improve the technique in \cite{dub} (see also a recent work \cite{AT} for other improvements of \cite{dub}) and give an analytic proof of Conjecture \ref{conj} in the context of a special $2$-truncated solution of PII, namely, the Hastings-McLeod solution \cite{HM}. This solution might be the most famous one among the Painlev\'{e} transcendents, due to its frequent appearances in applications, especially in mathematical physics. For instance, the cumulative distribution function of the celebrated Tracy-Widom distribution \cite{TW1,TW2} admits an integral representation involving the Hastings-McLeod solution. It is noted that the Tracy-Widom distribution is also applied to describe the length of the longest increasing subsequence in random permutations \cite{BDJ}. Another application is the appearance of $\Psi$ functions associated with Hastings-Mcleod solution in building new universality class of limiting kernel for certain critical unitary random matrix ensembles \cite{BI,CK}; see also \cite{FW} for a nice review of this aspect and \cite{DKZ,DG} for its more recent applications related to non-intersecting Brownian motions. Our main result is stated in the next section.

\section{Statement of results}
The Hastings-McLeod solution $y_{HM}$ is a special solution of \eqref{eq:p2a} with $\alpha=0$, i.e., it satisfies the equation
\begin{equation}
y''=2y^{3}+xy. \label{eq:p2}
\end{equation}
The solution $y_{HM}$ is known to be pole-free on the real axis (\cite{HM}), and has the following asymptotics:
\begin{equation*}
y_{HM}(x)\sim \left\{
                \begin{array}{ll}
                  \Ai(x), & ~~~~\hbox{as $x\to +\infty$,} \\
                  \sqrt{-x/2}, & ~~~~\hbox{as $x\to -\infty$,}
                \end{array}
              \right.
\end{equation*}
where $\Ai(x)$ denotes the usual Airy function \cite{DLMF}. A plot of $y_{HM}(x)$ for real $x$ is shown in the left picture of Figure \ref{fig:HM}. The locations of poles for $y_{HM}$ is illustrated in the right picture of Figure \ref{fig:HM}. The six dashed lines are the critical lines defined in \eqref{def:gammak}, and it is clear from the picture that all the poles are located in the sectors $\Xi_1 \cup \Xi_4$, which is consistent with Conjecture \ref{conj}.

\begin{figure}[ht]
\begin{center}
\resizebox*{7cm}{!}{\includegraphics{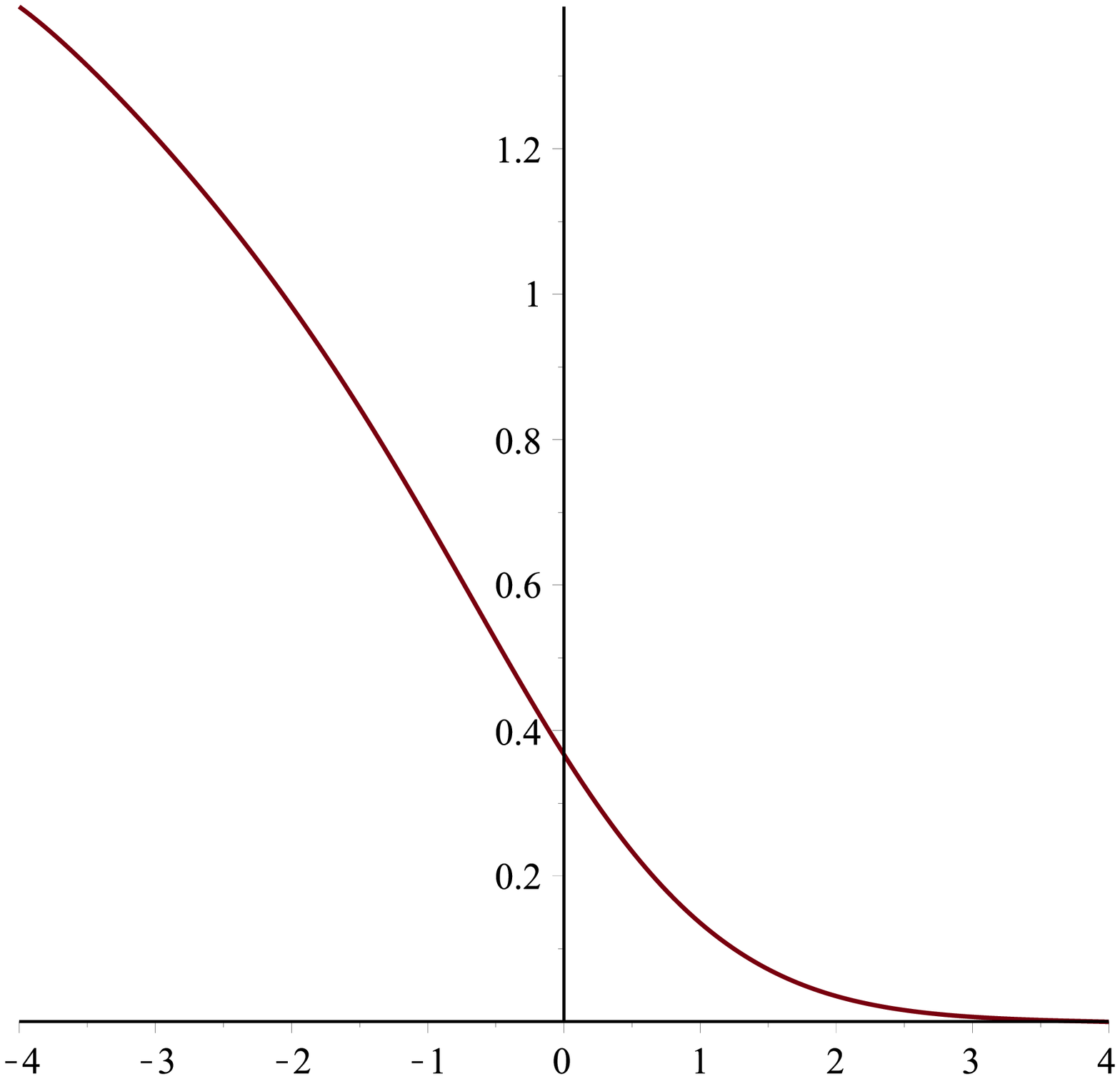}}
\hspace{2mm}
\resizebox*{7cm}{!}{\includegraphics{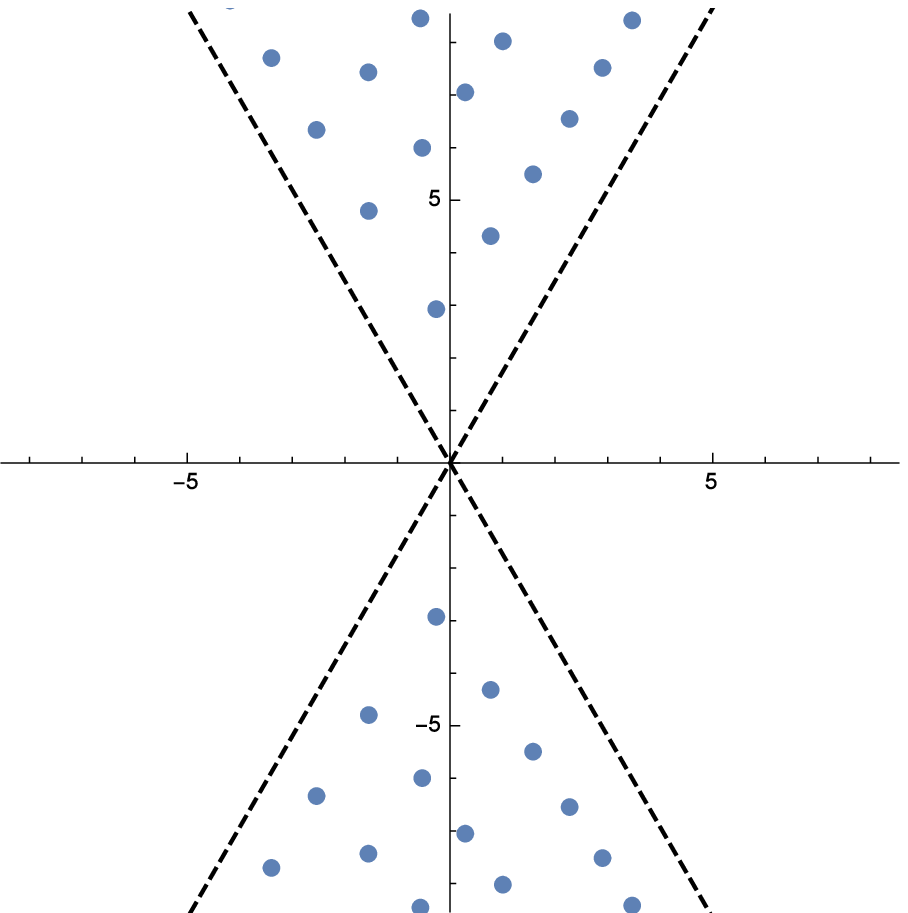}}
\caption{The Hastings-MeLeod solution (left) and its pole distribution (right).}
\label{fig:HM} \end{center}
\end{figure}

Our main result is stated as follows.
\begin{thm}\label{main-1}
The Hastings-McLeod solution $y_{HM}$ of the second Painlev\'{e} equation \eqref{eq:p2} is pole-free in
the region $\arg x \in [-\frac{\pi}{3},\frac{\pi}{3}]\cup [\frac{2\pi}{3},\frac{ 4 \pi}{3}]$.
\end{thm}

For $|x|$ large enough, a partial result was shown in \cite{itskapaev} via the Riemann-Hilbert approach; see also \cite[Theorems 11.1 and 11.7]{book}. A more recent progress toward this result was obtained by Bertola in \cite{bert}, where he showed that $y_{HM}$ is pole-free in the sector $\arg x \in [-\frac{\pi}{3},\frac{\pi}{3}]$. His proof is based on the representation of the Hastings-McLeod solution in terms of the second logarithmic derivative of the Fredholm determinant of a certain integral operator and an operator-norm estimate. In contrast, our method is based on a direct analysis of \eqref{eq:p2}, and it can be applied to other equations including the general PII equation \eqref{eq:p2a} with $\alpha\ne 0$.

\section{Strategy of proof}\label{sec:about}
Although our method works for both of the two sectors $\arg x \in [-\frac{\pi}{3},\frac{\pi}{3}]$ and $\arg x \in [\frac{2\pi}{3},\frac{ 4 \pi}{3}]$, we shall focus on the sector $\arg x \in [\frac{2\pi}{3},\frac{ 4 \pi}{3}]$ and briefly mention the ideas of proof for the sector $\arg x\ \in [-\frac{\pi}{3},\frac{\pi}{3}]$, as the desired result about this sector was already shown in \cite{bert}; see Section \ref{sec:proofs} below.

We first note that $\overline{y_{HM}}(\bar{z})$ is also a solution to (\ref{eq:p2}). This, together with the fact that $y_{HM}$ is real on the real line and uniqueness of the solution, implies that $y_{HM}(z)=\overline{y_{HM}}(\bar{z})$. Therefore, for $\arg x \in [\frac{2\pi}{3},\frac{ 4 \pi}{3}]$, it is sufficient to prove the following result:
\begin{thm}\label{main}
The Hastings-McLeod solution $y_{HM}$ is pole-free in the region
\begin{equation*}
\Omega:=\left\{x\in\mathbb{C}~\Big{|}~2\pi/3\leq\arg x\leq\pi\right\}.
\end{equation*}
\end{thm}

As mentioned before, we will use the same idea as in \cite{dub} to prove the Theorem. To be precise,
we will analyze $y_{HM}$ in two regions
\begin{equation}\label{def:omega0}
\Omega_{0}:=\left\{x\in\mathbb{C}~\Big{|}~|x|\geqslant\frac{3^{4/3}}{2},~~2\pi/3\leqslant\arg x\leqslant\pi\right\}
\end{equation}
and
\begin{equation}\label{def:omega2}
\Omega_{2}:=\left\{x\in\mathbb{C}~\Big{|}~|x|\leqslant9/4,~~2\pi/3\leqslant\arg x\leqslant\pi\right\}.
\end{equation}
In each region we will construct an explicit quasi-solution consisting of polynomials and exponential functions, and show that the difference between $y_{HM}$ and the quasi-solution is small in a suitable norm. This shows that $y_{HM}$ is pole-free in both $\Omega_{0}$ and $\Omega_{2}$, and hence in $\Omega\subseteq\Omega_{0}\cup\Omega_{2}$.

The main challenge of the proof is to find an effective quasi-solution approximation of the Hastings-McLeod solution which has sufficient accuracy for both small and large $|x|$. This requires comprehensive knowledge of the asymptotics of the solution near infinity. To this end, we mention the following asymtotics of Hastings-McLeod solution relevant to our proof (see \cite[Theorem 11.7]{book} and \cite{itskapaev}).
\begin{prop}\label{prop:asy}
Let $y_{HM}$ be the Hastings-McLeod solution of the second Painlev\'{e} equation \eqref{eq:p2}, then
\begin{equation}\label{eq:asy real}
y_{HM}(x)=\frac{1}{2\sqrt{\pi}}x^{-1/4}e^{-\frac{2}{3}x^{3/2}}\left(1+\mathcal{O}(x^{-3/4})\right)
\end{equation}
as $x\to+\infty$ and $\arg x =0$;
\begin{equation}\label{eq:asy complex}
y_{HM}(x)=
\sqrt{-x/2}\left(1+\mathcal{O}((-x)^{-3/2})\right)
+c_-(-x)^{-1/4}e^{-\frac{2\sqrt{2}}{3}(-x)^{3/2}}\left(1+\mathcal{O}(x^{-1/4})\right)
\end{equation}
as $x\to \infty$ and $\arg x \in [\frac{2\pi}{3},\frac{4\pi}{3})$,
where
\begin{equation*}\label{eq:stokeM}
c_-=\frac{i2^{-7/4}}{\sqrt{\pi}}.
\end{equation*}
\end{prop}
The constant $c_-$ is the so-called quasi-linear Stokes' multiplier, which reflects the quasi-linear Stokes phenomenon for the second Painlev\'{e} transecedent; see \cite{Cos98,itskapaev,Kap92} for more details. We emphasize two features of the asymptotics in Proposition \ref{prop:asy}:
\begin{itemize}
  \item The asymptotics \eqref{eq:asy complex} is valid along the critical line $\arg x =2\pi/3$ (i.e., the boundary of the relevant sector), where the asymptotics is oscillatory.
  \item The Hastings-McLeod solution $y_{HM}$ is characterized by either of the two asymptotic relations \eqref{eq:asy real} and \eqref{eq:asy complex}. Indeed, it suffices to specify the asymptotics just along the boundary rays; see \cite[Chapter 11]{book}.
\end{itemize}
As we shall see later, the construction of quasi-solutions in the regions away from the origin is based on these asymptotic behaviors.

The rest of this paper is organized as follows. The analysis of $y_{HM}$ in $\Omega_{0}$ is accomplished in Section \ref{sec:left}. To get a concrete estimate of the initial values of $y_{HM}$ at $0$, we will also need to study $y_{HM}$ along $[0,+\infty)$ before we are able to construct a quasi-solution in $\Omega_{2}$, which is carried out in Sections \ref{sec:farright} and \ref{sec:origin}. Analysis in $\Omega_{2}$ is accomplished in Section \ref{sec:left0}. We conclude this paper with the proofs of our main results in Section \ref{sec:proofs}.

\section{Analysis of $y_{HM}$ in the region $\Omega_{0}$}\label{sec:left}

We start with the analysis of the PII equation (\ref{eq:p2}) in the region $|x|\geq 3^{4/3}/2$, $2\pi/3\leqslant\arg x\leqslant\pi$. Our goal in this section is to prove the following result:
\begin{prop}\label{prop:farleft}
The Hastings-McLeod solution $y_{HM}$ is pole-free in the region $\Omega_{0}$, where $\Omega_0$ is defined in \eqref{def:omega0}.
\end{prop}
As mentioned before, we will prove the above result by constructing a pole-free quasi-solution to PII, and showing that the difference between the quasi-solution and the Hastings-McLeod solution is bounded. Our construction of this quasi-solution is motivated by asymptotic expansions with exponential sums studied in \cite{costin},
which suggests that we make the following change of variables:
\begin{equation}
t=\frac{2}{3}\sqrt{2}(-x)^{3/2};\; \quad y(x)=\frac{\sqrt[3]{3t}}{2}h(t).\label{eq:changev1}
\end{equation}
This brings (\ref{eq:p2}) into the normalized form
\begin{equation}\label{eq:hleft}
h''(t)+\frac{h'(t)}{t}+\frac{h(t)}{2}-\frac{h(t)}{9t^{2}}-\frac{1}{2}h(t)^{3}=0,
\end{equation}
and the region of interest in Proposition \ref{prop:farleft} corresponds to
\begin{equation*}
\Omega_{1}:=\left\{t\in\mathbb{C}~\Big{|}~|t|\geqslant3,-\pi/2\leqslant\arg t\leqslant0\right\},
\end{equation*}
in the new variable $t$.

Let $h_{HM}$ denote the solution of (\ref{eq:hleft}) corresponding to the Hastings-McLeod solution. In view of \eqref{eq:asy complex}, one naturally expects to have the decomposition
\begin{equation}\label{eq:decoh}
h_{HM}=h_{p}+h_{e},
\end{equation}
where $h_{p}$ is a solution of (\ref{eq:hleft}) with pure power series behavior near $-i\infty$ (i.e., with zero quasi-linear Stokes' multiplier), and $h_{e}$ is exponentially small near $\infty$. This is also consistent with the fact the Painlev\'{e} equations admit a one-parameter family of solution represented by the sum
\[
y=\textrm{(power series)+(exponential terms)},
\]
which was found by Boutroux \cite{Boutroux-a,Boutroux-b}, and particularly this includes $y_{HM}$ as a special case of PII.
Since we only need to prove $h_{HM}$ is pole-free in $\Omega_{1}$, we do not need to consider full expansions. Instead, we will only show the existence of a decomposition \eqref{eq:decoh} with
$$h_{p}\sim1-\frac{1}{9t^{2}} \quad \textrm{and} \quad h_{e}\sim\frac{\sqrt{2}\tilde{c}e^{-t}}{\sqrt{t}},$$
where $\tilde{c}$ is a constant related to the quasi-linear Stokes multiplier $c_-$; see \eqref{eq:hmasymp1} below.

\subsection{Existence and uniqueness of the power series solution $h_{p}$}

Recall the asymptotics of $y_{HM}$ in \eqref{eq:asy complex}, by (\ref{eq:changev1}), this corresponds to a solution $h\sim1$ of (\ref{eq:hleft}) in $\Omega_1$. Formal asymptotic analysis of (\ref{eq:hleft}) indicates that there should exist a solution $h(t)\sim1-\frac{1}{9t^{2}}$. We thus substitute $$h(t)=1-\frac{1}{9t^{2}}+\frac{h_{1}(t)}{\sqrt{t}}$$
into (\ref{eq:hleft}), and get the equation
\begin{align}\label{eq:h1eq}
&h_{1}''(t)-h_{1}(t)
 \nonumber \\
&=\frac{73}{162t^{7/2}}-\frac{1}{1458t^{11/2}}+\left(-\frac{17}{36t^{2}}+\frac{1}{54t^{4}}\right)h_{1}(t)+\left(\frac{3}{2\sqrt{t}}-\frac{1}{6t^{5/2}}\right)h_{1}(t)^{2}+\frac{h_{1}(t)^{3}}{2t}
\nonumber \\
&=:R_{h}(t,h(t)).
\end{align}

Inverting the differential operator on the left side of (\ref{eq:h1eq}), we get the integral equation
\begin{multline}\label{eq:h1eqint}
h_{1}(t)=\mathcal{T}_{1}(h(t)):=\mathcal{L}_{1}\left(R_{h}(t,h(t))\right)
\\
:=\frac{1}{2}\left(e^{t}\int_{\infty}^{t}e^{-s}R_{h}(s,h(s))ds-e^{-t}\int_{-i\infty}^{t}e^{s}R_{h}(s,h(s))ds\right),
\end{multline}
where the first integral is along a horizontal line, while the second one is along a vertical ray starting from $t$.
We intend to prove existence of a solution $h_{p}$ by showing that $\mathcal{T}_{1}$ is a contractive map in a suitable Banach space. The expressions of $R_{1}$ and $\mathcal{L}$ indicate that it is necessary to estimate generalized exponential integrals in the complex plane. For this purpose, we introduce the following inequalities, which will also be used later:
\begin{lem}\label{lem:ineq}
Assume $f$ is analytic in the right half plane with $|f(s)|\leqslant c/|s|^{n}$ where $c>0$, $n>1$, and $\re s \geqslant0$. For $t\in\Omega_1$, we have the estimates
\begin{align}
\left|\int_{+\infty}^{t}e^{-ms}f(s)ds\right|&\leqslant\frac{ce^{-m \re t}}{m|t|^{n}}, \qquad m>0,\label{eq:ineq1}
\\
\left|\int_{+\infty}^{t}e^{-ms}f(s)ds\right|&\leqslant\frac{ce^{-m \re t}}{(n-1)|t|^{n-1}}, \qquad m\geqslant 0,\label{eq:ineq2}
\\
\left|\int_{-i\infty}^{t}e^{ms}f(s)ds\right|&\leqslant\frac{c\sqrt{\pi}\Gamma\left(\frac{n}{2}-\frac{1}{2}\right)e^{m \re t}}{2\Gamma\left(\frac{n}{2}\right)|t|^{n-1}},
\qquad m\geqslant 0.\label{eq:ineq3}
\end{align}
\end{lem}
\begin{proof}
We write $t=a+bi$ where $a\geqslant0$ and $b\leqslant0$. To prove the first inequality, we note that since $f$ is analytic with at least $t^{-n}$ decay in the right half plane, we can rotate the integration path to a horizontal one, namely $s=a+u+bi$ with $u$ ranging from $\infty$ to 0. Then by direct calculations we have
\[
\left|\int_{+\infty}^{t}e^{-ms}f(s)ds\right|\leqslant ce^{-ma}\int_{0}^{+\infty}\frac{e^{-mu}}{((a+u)^{2}+b^{2})^{n/2}}du\leqslant\frac{ce^{-ma}}{(a^{2}+b^{2})^{n/2}}\int_{0}^{+\infty}e^{-mu}du=\frac{ce^{-ma}}{m|t|^{n}}.
\]

Alternatively, we can also rotate the integration path to a radial
one, which gives
\[
\left|\int_{+\infty}^{t}e^{-ms}f(s)ds\right|\leqslant ce^{-ma}\int_{|t|}^{+\infty}\frac{1}{|s|^{n}}d|s|=\frac{ce^{-m \re t}}{(n-1)|t|^{n-1}}.
\]

To prove the last inequality, we rotate the contour to a vertical one, namely $s=a+ui$ with $u$ ranging from $-\infty$ to 0. By direct calculations we have
\begin{multline*}
\left|\int_{-i\infty}^{t}e^{ms}f(s)ds\right|\leqslant c e^{ma}\int_{-\infty}^{0}\frac{1}{(a^{2}+(b+u)^{2})^{n/2}}du\leqslant c e^{ma}\int_{-\infty}^{0}\frac{1}{(a^{2}+b^{2}+u{}^{2})^{n/2}}du \\
\leqslant\frac{c e^{ma}}{|t|^{n-1}}\int_{-\infty}^{0}\frac{1}{(1+v{}^{2})^{n/2}}dv=\frac{ c\sqrt{\pi}\Gamma\left(\frac{n}{2}-\frac{1}{2}\right) e^{ma}}{2\Gamma\left(\frac{n}{2}\right)|t|^{n-1}}.
\end{multline*}

This completes the proof of the lemma.
\end{proof}

Now we are ready to prove the main result below.
\begin{prop}\label{prop:hp}
There is a unique solution of equation (\ref{eq:h1eq}) satisfying
\begin{equation*}
|h_{1}(t)|\leqslant\frac{6}{5|t|^{7/2}}
\end{equation*}
in $\Omega_{1}$.
\end{prop}
\begin{proof}
We will prove the proposition using the contraction map theorem in the Banach space $\mathcal{S}_{1}$ of analytic functions in the interior of $\Omega_{1}$, continuous up to the boundary, equipped with the weighted norm
\[
||f||_{1}=\sup_{t\in\Omega_{1}}\left|t^{7/2}f(t)\right|.
\]

We now show that the operator $\mathcal{T}_{1}$ (see (\ref{eq:h1eq}) and (\ref{eq:h1eqint})) is a contractive map in a ball of size $\frac{6}{5}$ of $\mathcal{S}_{1}$. Since $\mathcal{T}_{1}$ clearly preserves analyticity and continuity, we only need to show two statements, namely,
\begin{enumerate}
  \item[(i)] if $||f||_{1}\leqslant\frac{6}{5}$, then $||\mathcal{T}_{1}(f)||_{1}\leqslant\frac{6}{5}$;
  \item[(ii)] $||\mathcal{T}_{1}(f_{1})-\mathcal{T}_{1}(f_{2})||_{1}\leqslant\lambda||f_{1}-f_{2}||_{1}$
for some $\lambda<1$.
\end{enumerate}
These follow from direct calculations and elementary estimates using Lemma \ref{lem:ineq}.

\paragraph{Proof of statement (i):}
We now estimate $\mathcal{L}_{1}(R_{h}(t,f(t)))$ in (\ref{eq:h1eqint}), assuming $||f||_{1}\leqslant\frac{6}{5}$. The term $\frac{73}{162t^{7/2}}$ in (\ref{eq:h1eq}) needs special care due to its slow decay, and we estimate it using (\ref{eq:ineq1}), (\ref{eq:ineq3}), and integration by parts:
\[
\left|e^{t}\int_{+\infty}^{t}\frac{73e^{-s}}{162s^{7/2}}ds\right|\leqslant\frac{73}{162|t|^{7/2}},
\]
\[
\left|e^{-t}\int_{-i\infty}^{t}\frac{73e^{s}}{162s^{7/2}}ds\right|
\leqslant\frac{73}{162|t|^{7/2}}+\left|\frac{7e^{-t}}{2}\int_{-i\infty}^{t}\frac{73e^{s}}{162s^{9/2}}ds\right|
\leqslant\frac{73\left(7\sqrt{\pi}\Gamma\left(\frac{7}{4}\right)+4\Gamma\left(\frac{9}{4}\right)\right)}{648|t|^{7/2}\Gamma\left(\frac{9}{4}\right)}.
\]

The rest of the terms in (\ref{eq:h1eq}) are estimated by adding the absolute values of all monomials in $1/\sqrt{t}$ and using (\ref{eq:ineq2}) and (\ref{eq:ineq3}). In summary, we have
\begin{multline*}
\left|t^{7/2}\mathcal{L}_{1}\left(\frac{73}{162t^{7/2}}-\frac{1}{1458t^{11/2}}\right)\right| \\
\leqslant\frac{73}{162}+\frac{511\sqrt{\pi}\Gamma\left(\frac{7}{4}\right)}{1296\Gamma\left(\frac{9}{4}\right)}+\frac{1}{|t|}\left(\frac{1}{13122}+\frac{\sqrt{\pi}\Gamma\left(\frac{9}{4}\right)}{5832\Gamma\left(\frac{11}{4}\right)}\right)
<\frac{41}{40},
\end{multline*}
\begin{multline*}
\left|t^{7/2}\mathcal{L}_{1}\left(\left|-\frac{17}{36t^{2}}+\frac{1}{54t^{4}}\right|f(t)\right)\right| \\
\leqslant\frac{6}{5}\left(\frac{1}{|t|^{3}}\left(\frac{\sqrt{\pi}\Gamma\left(\frac{13}{4}\right)}{216\Gamma\left(\frac{15}{4}\right)}+\frac{1}{702}\right)+\frac{1}{|t|}\left(\frac{17\sqrt{\pi}\Gamma\left(\frac{9}{4}\right)}{144\Gamma\left(\frac{11}{4}\right)}+\frac{17}{324}\right)\right)<\frac{41}{500},
\end{multline*}
\begin{multline*}
\left|t^{7/2}\mathcal{L}_{1}\left(\left(\frac{3}{2\sqrt{t}}-\frac{1}{6t^{5/2}}\right)f(t)^{2}\right)\right|\\
\leqslant\frac{36}{25}\left(\frac{1}{|t|^{5}}\left(\frac{1}{102}+\frac{\sqrt{\pi}\Gamma\left(\frac{17}{4}\right)}{24\Gamma\left(\frac{19}{4}\right)}\right)+\frac{1}{|t|^{3}}\left(\frac{3}{26}+\frac{3\sqrt{\pi}\Gamma\left(\frac{13}{4}\right)}{8\Gamma\left(\frac{15}{4}\right)}\right)\right)<\frac{3}{100},
\end{multline*}
and
\begin{align*}
\left|t^{7/2}\mathcal{L}_{1}\left(\frac{1}{2t}f(t)^{3}\right)\right|
\leqslant\frac{216}{125|t|^{7}}\left(\frac{1}{42}+\frac{\sqrt{\pi}\Gamma\left(\frac{21}{4}\right)}{8\Gamma\left(\frac{23}{4}\right)}\right)<10^{-4}.
\end{align*}
Adding up the above bounds we see that $$\left|t^{7/2}\mathcal{L}_{1}(R_{h}(t,h(t)))\right|<\frac{41}{40}+\frac{41}{500}+\frac{3}{100}+10^{-4}=1.1371<\frac{6}{5}.$$

\paragraph{Proof of statement (ii):}  We only need to do similar estimates
for the nonlinear terms in (\ref{eq:h1eq}) using (\ref{eq:ineq2}) and
(\ref{eq:ineq3}), as well as the simple facts that
\[
\left|f_{1}^{2}-f_{2}^{2}\right|\leqslant2||f_{1}-f_{2}||_{1}\frac{6}{5|t|^{7}}, \qquad\left|f_{1}^{3}-f_{2}^{3}\right|\leqslant3||f_{1}-f_{2}||_{1}\frac{36}{25|t|^{21/2}}.
\]

Straightforward calculations give us
\begin{multline*}
\left|t^{7/2}\mathcal{L}_{1}\left(\left(-\frac{17}{36t^{2}}+\frac{1}{54t^{4}}\right)(f_{1}(t)-f_{2}(t))\right)\right|\\
\leqslant\frac{1}{|t|^{3}}\left(\frac{\sqrt{\pi}\Gamma\left(\frac{13}{4}\right)}{216\Gamma\left(\frac{15}{4}\right)}+\frac{1}{702}\right)+\frac{1}{|t|}\left(\frac{17\sqrt{\pi}\Gamma\left(\frac{9}{4}\right)}{144\Gamma\left(\frac{11}{4}\right)}+\frac{17}{324}\right)<\frac{7}{100}||f_1-f_2||_1,
\end{multline*}
\begin{multline*}
\left|t^{7/2}\mathcal{L}_{1}\left(\left(\frac{3}{2\sqrt{t}}-\frac{1}{6t^{5/2}}\right)\left(f_{1}^{2}(t)-f_{2}^{2}(t)\right)\right)\right|\\
\leqslant\frac{6}{5}\left(\frac{1}{|t|^{5}}\left(\frac{1}{51}+\frac{\sqrt{\pi}\Gamma\left(\frac{17}{4}\right)}{12\Gamma\left(\frac{19}{4}\right)}\right)+\frac{1}{|t|^{3}}\left(\frac{3}{13}+\frac{3\sqrt{\pi}\Gamma\left(\frac{13}{4}\right)}{4\Gamma\left(\frac{15}{4}\right)}\right)\right)<\frac{1}{20}||f_1-f_2||_1,
\end{multline*}
and
\[
\left|t^{7/2}\mathcal{L}_{1}\left(\frac{1}{2t}\left(f_{1}(t)^{3}-f_{2}(t)^{3}\right)\right)\right|\leqslant\frac{36}{25|t|^{7}}\left(\frac{1}{14}+\frac{\sqrt{\pi}\Gamma\left(\frac{21}{4}\right)}{8\Gamma\left(\frac{23}{4}\right)}\right)<3\cdot10^{-4}||f_1-f_2||_1.
\]

Adding up the above bounds we see that
$$||\mathcal{T}_{1}(f_{1})-\mathcal{T}_{1}(f_{2})||_{1}
<\left(\frac{7}{100}+\frac{1}{20}+3\cdot10^{-4}\right)||f_{1}-f_{2}||_{1}<\frac{7}{50}||f_{1}-f_{2}||_{1}.$$

This completes the proof of Proposition \ref{prop:hp}.
\end{proof}
Now we define
\begin{equation*}\label{eq:hpdef}
h_{p}(t):=1-\frac{1}{9t^{2}}+\frac{h_{2}(t)}{t^{4}}, \qquad h_{2}(t):=t^{7/2}h_{1}(t).
\end{equation*}
It follows from Proposition \ref{prop:hp} that $|h_{2}(t)|\leqslant\frac{6}{5}$.
Clearly $h_{p}$ is pole-free in $\Omega_{1}$.

\subsection{Existence and uniqueness of the exponential correction $h_{e}$}

To analyze the exponential part of $h_{HM}$, we see from \eqref{eq:asy complex} that
\begin{equation}\label{eq:hmasymp}
y_{HM}\sim\sqrt{-x/2}+c_{-}(-x)^{-1/4}e^{-\frac{2\sqrt{2}}{3}(-x)^{3/2}}
\end{equation}
for $x\rightarrow(-1+\sqrt{3}i)\infty$, which means by \eqref{eq:changev1} that,
\begin{equation}\label{eq:hmasymp1}
h_{HM}\sim1+\frac{\sqrt{2}\tilde{c}e^{-t}}{\sqrt{t}}, \qquad \tilde{c}=\frac{2^{3/4} c_-}{\sqrt{3}}=\frac{i}{2\sqrt{3\pi}},
\end{equation}
as $t\to -i\infty$.

Thus we write $$h=h_{p}+\frac{\tilde{c}e^{-t}}{\sqrt{t}}h_{3},$$
and substitute this expression into (\ref{eq:hleft}), which gives the equation
\begin{equation}\label{eq:h3eq}
h_{3}''(t)-2h_{3}'(t)
=\left(\frac{3}{2}h_{p}(t)^{2}-\frac{5}{36t^{2}}-\frac{3}{2}\right)h_{3}(t)+\frac{3\tilde{c}e^{-t}h_{p}(t)h_{3}(t)^{2}}{2\sqrt{t}}+\frac{\tilde{c}^{2}e^{-2t}h_{3}(t)^{3}}{2t}
\end{equation}
with
$$h_{3}\sim\sqrt{2}\quad \textrm{as} \quad t\rightarrow-i\infty.$$
Based on the first few terms of the asymptotic expansion of $h_{3}$, we construct a quasi-solution
\begin{equation}\label{eq:ha}
h_{a}(t)=\frac{\tilde{c}^{3}e^{-3t}}{2t^{3/2}}+\frac{\tilde{c}^{2}e^{-2t}}{\sqrt{2}t}-\frac{41\tilde{c}e^{-t}}{36t^{3/2}}+\frac{\tilde{c}e^{-t}}{\sqrt{t}}-\frac{17}{36\sqrt{2}t}+\sqrt{2},
\end{equation}
and our goal is to show that there exists a solution to \eqref{eq:h3eq} of the form
\begin{equation}\label{eq:h3deco}
h_{3}=h_{a}+\delta_{1},
\end{equation}
where $\delta_{1}$ is small in a suitable norm. The equation for $\delta_{1}$ can be found by substituting \eqref{eq:h3deco} into (\ref{eq:h3eq}), which gives
\begin{multline}\label{eq:delta1}
\delta_{1}''(t)-2\delta_{1}'(t)=-R_{1}(t)+\text{\ensuremath{\delta_{1}}}(t)\left(\frac{3\tilde{c}^{2}e^{-2t}h_{a}(t)^{2}}{2t}+\frac{3\tilde{c}e^{-t}h_{p}(t)h_{a}(t)}{\sqrt{t}}+\frac{3}{2}h_{p}(t)^{2}-\frac{5}{36t^{2}}-\frac{3}{2}\right)\\
+\delta_{1}(t)^{2}\left(\frac{3\tilde{c}^{2}e^{-2t}h_{a}(t)}{2t}+\frac{3\tilde{c}e^{-t}h_{p}(t)}{2\sqrt{t}}\right)+\frac{\tilde{c}^{2}e^{-2t}\delta_{1}(t)^{3}}{2t}=:R_{d}(\text{\ensuremath{\delta_{1}}}(t),t),
\end{multline}
where
\begin{multline}\label{eq:r1}
R_{1}(t)=h_{a}''(t)-2h_{a}'(t)-\left(\frac{3}{2}h_{p}(t)^{2}-\frac{5}{36t^{2}}-\frac{3}{2}\right)h_{a}(t)\\
-\frac{3\tilde{c}e^{-t}h_{p}(t)h_{a}(t)^{2}}{2\sqrt{t}}-\frac{\tilde{c}^{2}e^{-2t}h_{a}(t)^{3}}{2t}.
\end{multline}

We obtain the following integral equation by inverting the operator on the left side of (\ref{eq:delta1}):
\begin{equation}\label{eq:delta1int}
\text{\ensuremath{\delta_{1}}}(t)=\mathcal{T}_{2}(\text{\ensuremath{\delta_{1}}(t)}):=\mathcal{L}_{2}(R_{d}(\text{\ensuremath{\delta_{1}}}(t),t)):=\int_{\infty}^{t}e^{2u}\int_{\infty}^{u}e^{-2s}R_{d}(\text{\ensuremath{\delta_{1}}}(s),s)dsdu.
\end{equation}

To estimate $\mathcal{L}_{2}$, we introduce a lemma similar to Lemma \ref{lem:ineq}.
\begin{lem}\label{lem:ineq2}
Assume $f$ is analytic in the right half plane with $|f(t)|\leqslant c/|t|^{n}$ where $c>0$ and $\re t\geqslant 0$.
For $n>1$ and $m>0$, we have the estimate
\begin{equation}
\left|\mathcal{L}_{2}(e^{-mt}f(t))\right|\leqslant\frac{ce^{-m \re t}}{m(m+2)|t|^{n}}. \label{eq:ineq1-1}
\end{equation}
For $n>2$ and $m\geqslant0$, we have the estimate
\begin{equation}
\left|\mathcal{L}_{2}(e^{-mt}f(t))\right|\leqslant\frac{ce^{-m \re t}}{(n-1)(n-2)|t|^{n-2}}.\label{eq:ineq2-1}
\end{equation}
\end{lem}
\begin{proof}
Since $f$ is analytic, we can deform the integration paths into horizontal ones as in Lemma \ref{lem:ineq}. Denoting $t=a+bi$, where $a\geqslant0$ and $b\leqslant0$, we have
\begin{multline*}
\left|\mathcal{L}_{2}(e^{-mt}f(t))\right|\leqslant\left|\int_{\infty}^{a}e^{2u}\int_{\infty}^{u}e^{-(m+2)s}f(s+bi)dsdu\right|
\\
\leqslant\frac{c}{|t|^{n}}\left|\int_{\infty}^{a}e^{2u}\int_{\infty}^{u}e^{-(m+2)s}dsdu\right|=\frac{ce^{-ma}}{m(m+2)|t|^{n}}.
\end{multline*}

Alternatively, by deforming integration paths into radial ones, we have
\[
\mathcal{L}_{2}(e^{-mt}f(t))|\leqslant ce^{-ma}\int_{\infty}^{|t|}\int_{\infty}^{|u|}|f(s)|d|s|d|u|\leqslant\frac{ce^{-ma}}{(n-1)(n-2)|t|^{n-2}},
\]
which is \eqref{eq:ineq2-1}.
\end{proof}
We are then ready to prove
\begin{prop}\label{prop:delta1}
There is a unique solution of equation (\ref{eq:delta1})
satisfying $$|\delta_{1}(t)|\leqslant\frac{5}{2|t|^{2}}$$ in $\Omega_{1}$.
\end{prop}
\begin{proof}
The strategy is the same as in the proof of Proposition \ref{prop:hp}. We consider the Banach space $\mathcal{S}_{2}$ of analytic functions in $\Omega_{1}$, continuous up to the boundary, equipped with the weighted norm
\[
||f||_{2}=\sup_{t\in\Omega_{1}}\left|t^{2}f(t)\right|.
\]

We will prove that the operator $\mathcal{T}_{2}$ in (\ref{eq:delta1int}) is contractive in a ball of size $\frac{5}{2}$ of $\mathcal{S}_{2}$  with the help of Lemmas \ref{lem:ineq} and \ref{lem:ineq2}, by showing
\begin{enumerate}
  \item [(i)] if $||f||_{2}\leqslant\frac{5}{2}$, then $||\mathcal{T}_{2}(f)||_{2}\leqslant\frac{5}{2}$,
  \item [(ii)] $||\mathcal{T}_{2}(f_{1})-\mathcal{T}_{2}(f_{2})||_{2}\leqslant\lambda||f_{1}-f_{2}||_{2}$
for some $\lambda<1$.
\end{enumerate}

\paragraph{Proof of statement (i):}
We first estimate $R_{1}$ in (\ref{eq:r1}). Substituting the expression $$h_{p}(t)=1-\frac{1}{9t^{2}}+\frac{h_{2}(t)}{t^{4}}$$ with $h_{a}$ defined in (\ref{eq:ha}) into (\ref{eq:r1}),
we get an expression of the form
\[
R_{1}(t)=\sum_{k=4}^{11}\sum_{m=0}^{11}\frac{c_{k,m}e^{-mt}}{t^{k/2}}+h_2(t)\sum_{k=8}^{15}\sum_{m=0}^{7}\frac{\hat{c}_{k,m}e^{-mt}}{t^{k/2}}
+h_2(t)^2\sum_{k=16}^{19}\sum_{m=0}^{3}\frac{c_{k,m}e^{-mt}}{t^{k/2}},
\]
where $c_{k,m}$ and $\hat{c}_{k,m}$ are constants that can be written down explicitly, and in fact most of them are either zero or very small. For our purpose, it suffices to write out the terms with $k\leqslant7$ and estimate the rest crudely. Elementary calculations show that
$$R_{1}(t)=\tilde{R}_{1,1}(t)+\tilde{R}_{1,2}(t),$$
where
\begin{multline*}\label{eq:prince}
\tilde{R}_{1,1}(t)=-\frac{15\tilde{c}^{6}e^{-6t}}{2\sqrt{2}t^{3}}-\frac{29\tilde{c}^{5}e^{-5t}}{4t^{5/2}}+\frac{157\tilde{c}^{4}e^{-4t}}{12\sqrt{2}t^{3}}-\frac{6\sqrt{2}\tilde{c}^{4}e^{-4t}}{t^{2}}+\frac{359\tilde{c}^{3}e^{-3t}}{24t^{5/2}}+\frac{343\tilde{c}^{2}e^{-2t}}{288\sqrt{2}t^{3}}+\frac{47\tilde{c}^{2}e^{-2t}}{3\sqrt{2}t^{2}}\\
-\frac{9409\tilde{c}e^{-t}}{1728t^{5/2}}-\frac{27\tilde{c}^{7}e^{-7t}}{8t^{7/2}}
+\frac{33\tilde{c}^{5}e^{-5t}}{4t^{7/2}}-\frac{1927\tilde{c}^{3}e^{-3t}}{1728t^{7/2}}
-\frac{1609\tilde{c}e^{-t}}{324t^{7/2}}-\frac{1513}{1296\sqrt{2}t^{3}},
\end{multline*}
and
\begin{align}
\left|\tilde{R}_{1,2}(t)\right|&=\left|\sum_{k=8}^{11}\sum_{m=0}^{11}\frac{c_{k,m}e^{-mt}}{t^{k/2}}+h_2(t)\sum_{k=8}^{15}\sum_{m=0}^{7}\frac{\hat{c}_{k,m}e^{-mt}}{t^{k/2}}
+h_2(t)^2\sum_{k=16}^{19}\sum_{m=0}^{3}\frac{c_{k,m}e^{-mt}}{t^{k/2}}\right| \nonumber
\\
& \leqslant\sum_{k=8}^{11}\sum_{m=0}^{11}\frac{|c_{k,m}|}{|t|^{k/2}}+\frac{6}{5}\sum_{k=8}^{15}\sum_{m=0}^{7}\frac{|\hat{c}_{k,m}|}{|t|^{k/2}}
+\frac{36}{25}\sum_{k=16}^{19}\sum_{m=0}^{3}\frac{|c_{k,m}|}{|t|^{k/2}} \nonumber
\\
&<\frac{13}{10|t|^{9/2}}+\frac{6}{5|t|^{11/2}}+\frac{1}{10|t|^{13/2}}+\frac{1}{10|t|^{15/2}}+\frac{1}{2|t|^{17/2}}+\frac{1}{2|t|^{19/2}}
\nonumber \\
&~~~\qquad \qquad \qquad \qquad \qquad +\frac{4}{5|t|^{9}}+\frac{31}{10|t|^{8}}+\frac{1}{5|t|^{7}}+\frac{1}{|t|^{6}}+\frac{3}{2|t|^{5}}+\frac{26}{5|t|^{4}}. \nonumber 
\end{align}

To estimate $\mathcal{L}_{2}\left(\tilde{R}_{1,1}\right)$, we take the absolute value of each term in $\tilde{R}_{1,1}$, applying $\mathcal{L}_{2}$, and then adding them up. The last term $-\frac{1513}{1296\sqrt{2}t^{3}}$ in $\tilde{R}_{1,1}$ is special, and we use (\ref{eq:ineq1}) to estimate the inner integral and a radial path for the outer integral, which
gives
\begin{equation}
\left|\mathcal{L}_{2}\left(-\frac{1513}{1296\sqrt{2}t^{3}}\right)\right|
\leqslant\frac{1513}{2592\sqrt{2}}\int_{\infty}^{|t|}\frac{1}{|u|^{3}}d|u|\leqslant\frac{1513}{5184\sqrt{2}|t|^{2}}.\label{eq:special}
\end{equation}

For the other terms in $\mathcal{L}_{2}\left(\tilde{R}_{1,1}\right)$, we simply use (\ref{eq:ineq1-1}), which together with (\ref{eq:special}) implies
\begin{equation}\label{eq:l2r11}
\left|\mathcal{L}_{2}\left(\tilde{R}_{1,1}(t)\right)\right|<\frac{8}{25|t|^{5/2}}+\frac{7}{25|t|^{7/2}}+\frac{1}{250|t|^{3}}+\frac{1}{4|t|^{2}}.
\end{equation}

To estimate $\mathcal{L}_{2}\left(\tilde{R}_{1,2}\right)$, we use (\ref{eq:ineq2-1}) and obtain
\begin{multline}
\left|\mathcal{L}_{2}\left(\tilde{R}_{1,2}(t)\right)\right|<\frac{26}{175|t|^{5/2}}+\frac{8}{105|t|^{7/2}}+\frac{2}{495|t|^{9/2}}+\frac{2}{715|t|^{11/2}}+\frac{2}{195|t|^{13/2}}+\frac{2}{255|t|^{15/2}}\\
+\frac{1}{70|t|^{7}}+\frac{31}{420|t|^{6}}+\frac{1}{150|t|^{5}}+\frac{1}{20|t|^{4}}+\frac{1}{8|t|^{3}}+\frac{13}{15|t|^{2}}.\label{eq:l2r12}
\end{multline}

Combining (\ref{eq:l2r11})--(\ref{eq:l2r12}) and the fact that $|t|\geqslant 3$, we get
\begin{equation}\label{eq:ball0}
\left|\mathcal{L}_{2}\left(R_{1}(t)\right)\right|<\frac{8}{5|t|^{2}}.
\end{equation}

Now we assume $|f(t)|\leqslant\frac{5}{2|t|^{2}}$. We estimate the linear term in (\ref{eq:delta1}) in a similar way. Direct calculations using the definitions of $h_{p}$ and $h_{a}$ show that
\begin{equation*}\label{eq:hlinear}
\frac{3\tilde{c}^{2}e^{-2t}h_{a}(t)^{2}}{2t}+\frac{3\tilde{c}e^{-t}h_{p}(t)h_{a}(t)}{\sqrt{t}}+\frac{3}{2}h_{p}(t)^{2}-\frac{5}{36t^{2}}-\frac{3}{2}
=\tilde{R}_{1,3}(t)+\tilde{R}_{1,4}(t),
\end{equation*}
where
\[
\tilde{R}_{1,3}(t)=\frac{9\tilde{c}^{3}e^{-3t}}{\sqrt{2}t^{3/2}}+\frac{6\tilde{c}^{2}e^{-2t}}{t}
-\frac{17\tilde{c}e^{-t}}{12\sqrt{2}t^{3/2}}+\frac{3\sqrt{2}\tilde{c}e^{-t}}{\sqrt{t}},
\]
and
\[
\tilde{R}_{1,4}(t)=\sum_{k=4}^{8}\sum_{m=0}^{8}\frac{\hat{b}_{k,m}e^{-mt}}{t^{k/2}}+h_{2}(t)\sum_{k=8}^{12}\sum_{m=0}^{4}\frac{\hat{b}_{k,m}e^{-mt}}{t^{k/2}}+\frac{3h_{2}(t)^{2}}{2t^{8}}
\]
with $\hat{b}_{k,m}$ being certain constants. As before, after applying $\mathcal{L}_{2}$, we estimate $\tilde{R}_{1,3}$ using (\ref{eq:ineq1-1}), and $\tilde{R}_{1,4}$ by taking $|h_{2}(t)|\leqslant6/5$ and using (\ref{eq:ineq2-1}). This gives
\begin{equation}\label{eq:l2r13}
\left|t^{2}\mathcal{L}_{2}\left(\tilde{R}_{1,3}(t)f(t)\right)\right|
\leqslant\frac{5}{2}\left(\frac{3\tilde{c}^{3}}{5\sqrt{2}|t|^{3/2}}+\frac{3\tilde{c}^{2}}{4|t|}
+\frac{\sqrt{2}\tilde{c}}{|t|^{1/2}}+\frac{17\tilde{c}}{36\sqrt{2}|t|^{3/2}}\right),
\end{equation}
and
\begin{multline}\label{eq:l2r14}
\left|t^{2}\mathcal{L}_{2}\left(\tilde{R}_{1,4}(t)f(t)\right)\right|<\frac{5}{2}\bigg(\frac{3}{2000|t|^{3/2}}+\frac{1}{25|t|^{5/2}}+\frac{7}{1000|t|^{7/2}}+\frac{3}{100|t|^{6}}
+\frac{1}{50|t|^{4}}\\
+\frac{1}{250|t|^{3}}+\frac{19}{100|t|^{2}}+\frac{7}{10000|t|}+\frac{13}{1000\sqrt{|t|}}+\frac{11}{100}\bigg).
\end{multline}

Thus,
\begin{equation}\label{eq:balllin}
\left|t^{2}\mathcal{L}_{2}\left(\left(\tilde{R}_{1,3}(t)+\tilde{R}_{1,4}(t)\right)f(t)\right)\right|
<\frac{3}{10}\frac{5}{2}=\frac{3}{4}.
\end{equation}

The quadratic part of (\ref{eq:delta1int}) is also estimated using (\ref{eq:ineq2-1}). We have
\begin{align}\label{eq:ballquad}
& \left|t^{2}\mathcal{L}_{2}\left(\left|\frac{3\tilde{c}^{2}e^{-2t}h_{a}(t)}{2t}+\frac{3\tilde{c}e^{-t}h_{p}(t)}{2\sqrt{t}}\right|f(t)^{2}\right)\right|
\nonumber \\
& \leqslant \frac{25}{4}\bigg(\frac{12|\tilde{c}|}{325t^{9/2}}+\frac{2|\tilde{c}|}{297|t|^{5/2}}+\frac{41|\tilde{c}|^{3}}{594|t|^{5/2}}+\frac{|\tilde{c}|^{5}}{33|t|^{5/2}}+\frac{17|\tilde{c}|^{2}}{480\sqrt{2}|t|^{2}}
\nonumber
\\
& \qquad \qquad \qquad \qquad \qquad ~~~~~  +\frac{3|\tilde{c}|^{4}}{40\sqrt{2}|t|^{2}}+\frac{2|\tilde{c}|^{3}}{21|t|^{3/2}}+\frac{|\tilde{c}|^{2}}{4\sqrt{2}|t|}+\frac{6|\tilde{c}|}{35\sqrt{|t|}}\bigg)
\nonumber
\\
& <\frac{1}{50}\frac{25}{4}=\frac{1}{8}.
\end{align}

Finally we estimate the cubic term of (\ref{eq:delta1int}) by (\ref{eq:ineq2-1}):
\begin{equation}\label{eq:balltri}
\left|t^{2}\mathcal{L}_{2}\left(\frac{\tilde{c}^{2}e^{-2t}}{2t}f(t)^{3}\right)\right|\leqslant\frac{1}{720\pi t^{3}}\left(\frac{5}{2}\right)^{3}<3\cdot10^{-4}.
\end{equation}

Therefore, combing the results in (\ref{eq:ball0}), (\ref{eq:balllin}),
(\ref{eq:ballquad}), and (\ref{eq:balltri}), we have
\[
||f(t)||_{2}<\frac{5}{2}\Rightarrow\left\Vert \mathcal{L}_{2}(f(t))\right\Vert _{2}<\frac{5}{2}.
\]

\paragraph{Proof of statement (ii):}  To estimate the linear part of (\ref{eq:delta1int}), we still use (\ref{eq:l2r13}) and (\ref{eq:l2r14}), which gives
\begin{equation}\label{eq:hcontra}
\left|t^{2}\mathcal{L}_{2}\left(\left(\tilde{R}_{1,3}(t)+\tilde{R}_{1,4}(t)\right)(f_{1}(t)-f_{2}(t))\right)\right|
<\frac{3}{10}||f_{1}-f_{2}||_{2}.
\end{equation}

For the nonlinear parts, we use
\[
\left|f_{1}^{2}(t)-f_{2}^{2}(t)\right|\leqslant2||f_{1}-f_{2}||_{2}\frac{5}{2|t|^{4}},\qquad \left|f_{1}^{3}(t)-f_{2}^{3}(t)\right|\leqslant3||f_{1}-f_{2}||_{2}\frac{25}{4|t|^{6}},
\]
which gives us
\begin{align}\label{eq:ballquad-1}
&\left|t^{2}\mathcal{L}_{2}\left(\left(\frac{3\tilde{c}^{2}e^{-2t}h_{a}(t)}{2t}+\frac{3\tilde{c}e^{-t}h_{p}(t)}{2\sqrt{t}}\right)
\left(f_{1}^{2}(t)-f_{2}^{2}(t)\right)\right)\right|
\nonumber
\\
&<
5||f_{1}-f_{2}||_{2}\bigg(\frac{12|\tilde{c}|}{325t^{9/2}}+\frac{2|\tilde{c}|}{297|t|^{5/2}}+\frac{41|\tilde{c}|^{3}}{594|t|^{5/2}}+\frac{|\tilde{c}|^{5}}{33|t|^{5/2}}+\frac{17|\tilde{c}|^{2}}{480\sqrt{2}|t|^{2}}+\frac{3|\tilde{c}|^{4}}{40\sqrt{2}|t|^{2}}
\nonumber \\
&\qquad \qquad \qquad \qquad \qquad ~~~~~~~~~~+\frac{2|\tilde{c}|^{3}}{21|t|^{3/2}}+\frac{|\tilde{c}|^{2}}{4\sqrt{2}|t|}+\frac{6|\tilde{c}|}{35\sqrt{|t|}}\bigg)
\nonumber
\\
&<\frac{5}{50} ||f_{1}(t)-f_{2}(t)||_{2}=\frac{1}{10}||f_{1}(t)-f_{2}(t)||_{2},
\end{align}
and
\begin{multline}\label{eq:balltri-1}
\left|t^{2}\mathcal{L}_{2}\left(\left|\frac{\tilde{c}^{2}e^{-2t}}{2t}\right|\left(f_{1}^{3}(t)-f_{2}^{3}(t)\right)\right)\right|
\\
\leqslant
\frac{1}{240 \pi  |t|^3}\frac{25}{4}||f_{1}(t)-f_{2}(t)||_{2}<4\cdot10^{-4}||f_{1}(t)-f_{2}(t)||_{2}.
\end{multline}

Combining the results in (\ref{eq:hcontra}), (\ref{eq:ballquad-1}), and (\ref{eq:balltri-1}), we get
\[
\left\Vert \mathcal{L}_{2}(f_{1}(t))-\mathcal{L}_{2}(f_{2}(t))\right\Vert _{2}<\frac{1}{2}\left\Vert f_{1}(t)-f_{2}(t)\right\Vert_2.
\]
The conclusion of the proposition then follows from the contraction mapping principle.
\end{proof}

\subsection{Proof of Proposition \ref{prop:farleft}}
\begin{proof}
We define $$h_{e}(t)=\frac{\tilde{c}e^{-t}}{\sqrt{t}}(h_{a}(t)+\delta_{1}(t)).$$
By Proposition \ref{prop:delta1}, it is clear that $h_e$ is pole-free in $\Omega_{1}$ and $h_{e}(t)\sim \frac{\sqrt{2}\tilde{c}e^{-t}}{\sqrt{t}}$ for large $|t|$. We have now obtained a solution with the decomposition $h=h_{p}+h_{e}$, which implies that $h$ has the asymptotic behavior (\ref{eq:hmasymp1}),
corresponding to the asymptotics of the Hastings-McLeod solution in (\ref{eq:hmasymp}). Since it is known \cite{book} that the Hastings-McLeod solution is the only solution having this asymptotic behavior (see also the comments after Proposition \ref{prop:asy}), we see that $$h_{HM}=h_{p}+h_{e},$$ which implies by Propositions \ref{prop:hp} and \ref{prop:delta1} that $h_{HM}$ is pole-free in $\Omega_{1}$. By (\ref{eq:changev1}), this means $y_{HM}$ is pole-free in $\Omega_{0}$.
\end{proof}

\section{Analysis of $y_{HM}$ for $x\geqslant3$}\label{sec:farright}

Proposition \ref{prop:farleft} implies that $y_{HM}$ stays close to its truncated asymptotic expansion in $\Omega_{0}$. However, when $|x|$ becomes small, no asymptotic expansion can provide sufficient information about $y_{HM}$. Instead, we will have to reply on other methods, such as Taylor series and/or fitting numerical data, to build quasi-solutions with controlled error bounds. This requires knowledge of the initial value of $y_{HM}$ at a finite point. Although our previous result $h=h_{p}+h_{e}$ can provide initial conditions of $y_{HM}$, say, at $x=-\frac{3^{4/3}}{2}$, the error bound is much larger than $10^{-3}$, which is not sufficient. Instead, we will obtain an accurate initial condition at 0 using the asymptotic expansion of $y_{HM}$ at $+\infty$. On account of \eqref{eq:asy real}, we expect the asymptotic expansion to provide very accurate information of $y_{HM}$ even for relatively small $|x|$. Since the exponent in \eqref{eq:asy real} is different from that of (\ref{eq:hmasymp}), we need to use a different change of variable, namely,
\begin{equation}\label{eq:change2}
t=\frac{2}{3}x{}^{3/2};\;\qquad y(x)=\left(\frac{2}{3}\right)^{1/6}t^{1/3}h(t)
\end{equation}
to bring (\ref{eq:p2}) to the normalized form
\begin{equation}\label{eq:hright}
h''(t)+\frac{h'(t)}{t}-h(t)-\frac{h(t)}{9t^{2}}-\frac{4}{3}h(t)^{3}=0.
\end{equation}

Substituting $h(t)=\frac{e^{-t}}{2\sqrt{\pi}\sqrt{t}}h_{4}(t)$ into
(\ref{eq:hright}), we get
\begin{equation}\label{eq:hright1}
h_{4}''(t)-2h_{4}'(t)+\frac{5h_{4}(t)}{36t^{2}}-\frac{e^{-2t}h_{4}(t)^{3}}{3\pi t}=0.
\end{equation}
Based on asymptotic analysis of (\ref{eq:hright1}), we construct a quasi-solution given by
\begin{equation}\label{eq:hbdef}
h_{b}(t)=1-\frac{85085}{2239488t^{3}}+\frac{385}{10368t^{2}}-\frac{5}{72t}+\frac{e^{-2t}}{24\pi t}.
\end{equation}

Now we insert $$h_{4}(t)=h_{b}(t)+\delta_{2}(t)$$ into (\ref{eq:hright1}), and get the equation
\begin{multline}\label{eq:delta2}
\delta_{2}''(t)-2\delta_{2}'(t)=\delta_{2}(t)\left(\frac{e^{-2t}h_{b}(t)^{2}}{\pi t}-\frac{5}{36t^{2}}\right)+\frac{e^{-2t}h_{b}(t)\delta_{2}(t)^{2}}{\pi t}+\frac{e^{-2t}\delta_{2}(t)^{3}}{3\pi t}-R_{2}(t) \\=:R_{s}(\text{\ensuremath{\delta_{2}}}(t),t),
\end{multline}
where
\begin{equation}\label{eq:r2}
R_{2}(t)=h_{b}''(t)-2h_{b}'(t)+\frac{5h_{b}(t)}{36t^{2}}-\frac{e^{-2t}h_{b}(t)^{3}}{3\pi t}.
\end{equation}

We write (\ref{eq:delta2}) in integral form
\begin{equation}\label{eq:delta2int}
\delta_{2}(t)=\mathcal{T}_{3}(\text{\ensuremath{\delta_{2}}(t)})
:=\mathcal{L}_{2}(R_{s}(\text{\ensuremath{\delta_{2}}}(t),t)),
\end{equation}
where $\mathcal{L}_{2}$ is the same as defined in (\ref{eq:delta1int}).

The main result of this section is the following:
\begin{prop}\label{prop:delta2}
There is a unique solution of equation (\ref{eq:delta2})
satisfying
\begin{equation*}
|\delta_{2}(t)|\leqslant\frac{1}{80t{}^{2}}, \qquad |\delta_{2}'(t)|\leqslant\frac{11}{1000t^{2}},
\end{equation*}
for $t\geqslant2\sqrt{3}$.\end{prop}
\begin{proof}
The proof is very similar to that of Proposition \ref{prop:delta1}, but simpler since there is no power series part and $t$ is real positive. We consider the Banach space $\mathcal{S}_{3}$ of continuous functions in $[2\sqrt{3},+\infty)$ equipped with the weighted norm
\[
||f||_{3}=\sup_{t\geqslant2\sqrt{3}}\left|t^{2}f(t)\right|.
\]
We will prove that the operator $\mathcal{T}_{3}$ in (\ref{eq:delta2int}) is contractive in a ball of size $\frac{1}{80}$ of $\mathcal{S}_{3}$ using Lemmas \ref{lem:ineq} and \ref{lem:ineq2}, by showing
\begin{enumerate}
  \item [(i)] if $||f||_{3}\leqslant\frac{1}{80}$, then $||\mathcal{T}_{3}(f)||_{3}\leqslant\frac{1}{80}$,
  \item [(ii)] $||\mathcal{T}_{3}(f_{1})-\mathcal{T}_{3}(f_{2})||_{3}\leqslant\lambda||f_{1}-f_{2}||_{3}$
for some $\lambda<1$.
\end{enumerate}
Note that the estimates (\ref{eq:ineq1}), (\ref{eq:ineq2}), (\ref{eq:ineq1-1}) and (\ref{eq:ineq2-1}) are obviously true for continuous functions on the real line.

\paragraph{Proof of statement (i):}  We first estimate $R_{2}$ in (\ref{eq:r2}). Substituting the definition of $h_{b}$ given in (\ref{eq:hbdef}) into (\ref{eq:r2}), we get an expression of the form $$R_{2}(t)=\tilde{R}_{2,1}(t)+\tilde{R}_{2,2}(t),$$
where
\begin{equation}\label{eq:r21tt}
\tilde{R}_{2,1}(t)=\frac{163e^{-2t}}{3456\pi t^{3}}+\frac{5e^{-4t}}{864\pi^{2}t^{3}}-\frac{e^{-6t}}{576\pi^{3}t^{3}}-\frac{e^{-4t}}{24\pi^{2}t^{2}}+\frac{23e^{-2t}}{72\pi t^{2}},
\end{equation}
and
\begin{multline}\label{eq:r22tt}
\left|\tilde{R}_{2,2}(t)\right|=\left|\sum_{k=4}^{10}\sum_{m=0}^{8}\frac{d_{k,m}e^{-mt}}{t^{k}}\right|\leqslant\sum_{k=4}^{10}\sum_{m=0}^{8}\frac{|d_{k,m}|e^{-2\sqrt{3}m}}{t^{k}}\\
<\frac{12}{25t^{5}}+10^{-6}\left(\frac{3}{500t^{10}}+\frac{1}{50t^{9}}+\frac{1}{20t^{8}}+\frac{3}{5t^{7}}+\frac{1}{t^{6}}+\frac{15}{t^{4}}\right).
\end{multline}

By (\ref{eq:ineq1-1}), we obtain the estimate
\begin{equation}\label{eq:l2r21}
\left|t^{2}\mathcal{L}_{2}\left(\tilde{R}_{2,1}(t)\right)\right|\leqslant\frac{e^{-4t}}{576\pi^{2}}+\frac{163e^{-2t}}{27648\pi t}+\frac{5e^{-4t}}{20736\pi^{2}t}+\frac{e^{-6t}}{27648\pi^{3}t}+\frac{23e^{-2t}}{576\pi},
\end{equation}
and from (\ref{eq:ineq2-1}), we have the estimate
\begin{equation}\label{eq:l2r22}
\left|t^{2}\mathcal{L}_{2}\left(\tilde{R}_{2,2}(t)\right)\right|<\frac{1}{25t}
+10^{-7}\left(\frac{1}{1200t^{6}}+\frac{1}{280t^{5}}+\frac{1}{84t^{4}}+\frac{1}{5t^{3}}+\frac{1}{2t^{2}}+25\right).
\end{equation}

Combining (\ref{eq:l2r21}) and (\ref{eq:l2r22}) and note that $t\geqslant2\sqrt{3}$,
we get
\begin{equation}\label{eq:ballr0}
\left|t^{2}\mathcal{L}_{2}(R_{2}(t))\right|<\frac{3}{250}.
\end{equation}

Now we assume $||f(t)||_{3}\leqslant1/80$. To analyze the linear term in (\ref{eq:delta2int}), we simply use the definition of $h_{b}$ in (\ref{eq:hbdef}) to obtain
\[
\frac{e^{-2t}h_{b}(t)^{2}}{\pi t}-\frac{5}{36t^{2}}=\tilde{R}_{2,3}(t)+\tilde{R}_{2,4}(t),
\]
where
\begin{equation}\label{eq:r23def}
\tilde{R}_{2,3}(t)=-\frac{5}{36t^{2}}-\frac{5e^{-2t}}{36\pi t^{2}}+\frac{e^{-4t}}{12\pi^{2}t^{2}}+\frac{e^{-2t}}{\pi t},
\end{equation}
and
\[
|\tilde{R}_{2,4}(t)|=\left|\sum_{k=3}^{7}\sum_{m=2}^{4}\frac{\tilde{d}_{k,m}e^{-mt}}{t^{k}}\right|
<10^{-6}\left(\frac{1}{2t^{7}}+\frac{1}{t^{6}}+\frac{11}{5t^{5}}+\frac{51}{2t^{4}}+\frac{25}{t^{3}}\right).
\]

For the term $-\frac{5}{36t^{2}}$ in (\ref{eq:r23def}), we use (\ref{eq:ineq1}) once and (\ref{eq:ineq2}) once to obtain
\[
\left|t^{2}\mathcal{L}_{2}\left(\left(\frac{5}{36t^{2}}f(t)\right)\right)\right|\leqslant\frac{1}{80}\frac{5}{216t}.
\]

Using (\ref{eq:ineq1-1}) to estimate the rest of the terms in (\ref{eq:r23def}),
we obtain
\begin{equation}\label{eq:l2r23}
\left|t^{2}\mathcal{L}_{2}\left(\tilde{R}_{2,3}(t)f(t)\right)\right|\leqslant\frac{1}{80}\left(\frac{5e^{-2t}}{288\pi t^{2}}+\frac{e^{-4t}}{288\pi^{2}t^{2}}+\frac{e^{-2t}}{8\pi t}+\frac{5}{216t}\right).
\end{equation}
By (\ref{eq:ineq2-1}), we obtain the estimate
\begin{equation}\label{eq:l2r24}
\left|t^{2}\mathcal{L}_{2}\left(\tilde{R}_{2,4}(t)f(t)\right)\right|
<\frac{10^{-6}}{80}\left(\frac{1}{112t^{5}}+\frac{1}{42t^{4}}+\frac{11}{150t^{3}}+\frac{51}{40t^{2}}+\frac{25}{12t}\right).
\end{equation}

Combining (\ref{eq:l2r23}) and (\ref{eq:l2r24}), we get
\begin{equation}\label{eq:ballr1}
\left|t^{2}\mathcal{L}_{2}\left(\left(\frac{e^{-2t}h_{b}(t)^{2}}{\pi t}-\frac{5}{36t^{2}}\right)f(t)\right)\right|<\frac{1}{80}\frac{1}{125}=10^{-4}.
\end{equation}

The quadratic and cubic terms of (\ref{eq:delta2int}) can be estimated using (\ref{eq:ineq2-1}). We have
\begin{equation}\label{eq:ballr2}
\left|t^{2}\mathcal{L}_{2}\left(\frac{e^{-2t}h_{b}(t)}{\pi t}f(t)^{2}\right)\right|<10^{-6}\left(\frac{e^{-2t}}{20t^{4}}+\frac{7e^{-2t}}{100t^{3}}+\frac{e^{-2t}}{5t^{2}}+\frac{e^{-4t}}{25t^{2}}+\frac{5e^{-2t}}{t}\right)<2\cdot10^{-9}
\end{equation}
and
\begin{equation}\label{eq:ballr3}
t^{2}\mathcal{L}_{2}\left(\frac{e^{-2t}f(t)^{3}}{3\pi t}\right)\leqslant\frac{e^{-2t}}{46080000\pi t^{3}}<10^{-12}.
\end{equation}

Combining (\ref{eq:ballr0}), (\ref{eq:ballr1}), (\ref{eq:ballr2}), and (\ref{eq:ballr3}), we find
\[
||f(t)||_{3}\leqslant\frac{1}{80}\Rightarrow||\mathcal{T}_{3}(f(t))||_{3}\leqslant\frac{1}{80}.
\]

\paragraph{Proof of statement (ii):}  The estimate for the linear part of
(\ref{eq:delta2int}) is very similar to (\ref{eq:ballr1}), which
gives
\begin{equation}\label{eq:contrar1}
\left|t^{2}\mathcal{L}_{2}\left(\left(\frac{e^{-2t}h_{b}(t)^{2}}{\pi t}-\frac{5}{36t^{2}}\right)(f_{1}(t)-f_{2}(t))\right)\right|
<\frac{1}{125}||f_{1}(t)-f_{2}(t)||_{3}.
\end{equation}

For the nonlinear parts, we use
\[
\left|f_{1}^{2}-f_{2}^{2}\right|\leqslant2||f_{1}-f_{2}||_{3}\frac{1}{80t{}^{4}},\qquad \left|f_{1}^{3}-f_{2}^{3}\right|\leqslant3||f_{1}-f_{2}||_{3}\frac{1}{6400t^{6}},
\]
and (\ref{eq:ineq2-1}) to obtain the estimates
\begin{align}
&\left|t^{2}\mathcal{L}_{2}\left(\frac{e^{-2t}h_{b}(t)}{\pi t}\left(f_{1}(t)^{2}-f_{2}(t)^{2}\right)\right)\right|
\nonumber
\\
&<10^{-5}\left(\frac{4e^{-2t}}{5t^{4}}+\frac{e^{-2t}}{t^{3}}+\frac{3e^{-2t}}{t^{2}}+\frac{3e^{-4t}}{5t^{2}}+\frac{80e^{-2t}}{t}\right)\left\Vert f_{1}(t)-f_{2}(t)\right\Vert _{3}
\nonumber \\
&<10^{-6}\left\Vert f_{1}(t)-f_{2}(t)\right\Vert _{3}, \label{eq:contrar2}
\end{align}
\begin{equation}\label{eq:contrar3}
\left|t^{2}\mathcal{L}_{2}\left(\frac{e^{-2t}}{3\pi t}(f_{1}(t)^{3}-f_{2}(t)^{3})\right)\right|\leqslant\frac{e^{-2t}\left\Vert f_{1}(t)-f_{2}(t)\right\Vert _{3}}{192000\pi t^{3}}<10^{-10}\left\Vert f_{1}(t)-f_{2}(t)\right\Vert _{3}.
\end{equation}

Combining (\ref{eq:contrar1}), (\ref{eq:contrar2}) and (\ref{eq:contrar3}),
we see that
\[
\left\Vert \mathcal{L}_{2}(f_{1}(t))-\mathcal{L}_{2}(f_{2}(t))\right\Vert _{3}<\frac{1}{120}\left\Vert f_{1}(t)-f_{2}(t)\right\Vert _{3}.
\]

To estimate $\delta_{2}'(t)$, we first differentiate (\ref{eq:delta2int})
once to get
\[
\delta_{2}'(t)=\mathcal{L}'_{2}\left(R_{s}(\text{\ensuremath{\delta_{2}}}(t),t)\right)
=e^{2t}\int_{\infty}^{t}e^{-2s}R_{s}(\text{\ensuremath{\delta_{2}}}(s),s))ds.
\]

The integral can be estimated by using (\ref{eq:ineq1}) and (\ref{eq:ineq2}).
In particular, from (\ref{eq:r21tt}) and (\ref{eq:r22tt}), we get
\begin{multline}\label{eq:ldr2}
\left|t^{2}\mathcal{L}'_{2}(R_{2}(t))\right|<\frac{e^{-4t}}{144\pi^{2}}+\frac{163e^{-2t}}{13824\pi t}+\frac{5e^{-4t}}{5184\pi^{2}t}+\frac{e^{-6t}}{4608\pi^{3}t}+\frac{23e^{-2t}}{288\pi}\\
+\frac{3}{25t^{2}}+10^{-7}\left(\frac{1}{150t^{7}}+\frac{1}{40t^{6}}+\frac{1}{14t^{5}}+\frac{1}{t^{4}}+\frac{2}{t^{3}}+\frac{50}{t}\right)<\frac{21}{2000}.
\end{multline}

The rest can be estimated crudely. First we take absolute value of each term in (\ref{eq:hbdef}) to get a bound $|h_{b}|<21/20$. This means
\begin{multline*}
\left|\delta_{2}(t)\left(\frac{e^{-2t}h_{b}(t)^{2}}{\pi t}-\frac{5}{36t^{2}}\right)+\frac{e^{-2t}h_{b}(t)\delta_{2}(t)^{2}}{\pi t}+\frac{e^{-2t}\delta_{2}(t)^{3}}{3\pi t}\right|
\\
<3\cdot10^{-7}\frac{e^{-2t}}{t^{7}}+\frac{6\cdot10^{-5}e^{-2t}}{t^{5}}+\frac{1}{576t^{4}}+\frac{e^{-2t}}{200t^{3}}.
\end{multline*}
Thus, by (\ref{eq:ineq2}), we have
\begin{multline}\label{eq:ldrest}
\left|t^{2}\mathcal{L}'_{2}\left(\delta_{2}(t)\left(\frac{e^{-2t}h_{b}(t)^{2}}{\pi t}-\frac{5}{36t^{2}}\right)+\frac{e^{-2t}h_{b}(t)\delta_{2}(t)^{2}}{\pi t}+\frac{e^{-2t}\delta_{2}(t)^{3}}{3\pi t}\right)\right|
\\
<10^{-7}\frac{e^{-2t}}{2t^{4}}+\frac{3\cdot10^{-5}e^{-2t}}{2t^{2}}+\frac{1}{1728t}+\frac{e^{-2t}}{400}<18\cdot10^{-5}.
\end{multline}

Therefore, from (\ref{eq:ldr2}) and (\ref{eq:ldrest}), we have
\[
\left|t^{2}\mathcal{L}'_{2}(R_{2}(\delta_{2}(t),t))\right|<\frac{21}{2000}+18\cdot10^{-5}<\frac{11}{1000}.
\]
This completes the proof of Proposition \ref{prop:delta2}.
\end{proof}

As a consequence of the above proposition, it follows
\begin{cor}\label{yhm3}
For $y_{HM}$ we have the estimates
\[
\left|y_{HM}(3)-\frac{4}{607}\right|<8\cdot10^{-6},\qquad \left|y_{HM}'(3)+\frac{64}{5375}\right|<24\cdot10^{-6}.
\]
\end{cor}
\begin{proof}
We first note that since $y_{HM}$ is the unique solution satisfying \eqref{eq:asy real}, Proposition \ref{prop:delta2} and (\ref{eq:change2})
implies
\[
y_{HM}(x)=
\frac{e^{-\frac{2}{3}x^{3/2}}}{2\sqrt{\pi}x^{1/4}}\left(h_{b}\left(\frac{2}{3}x{}^{3/2}\right)
+\delta_{2}\left(\frac{2}{3}x{}^{3/2}\right)\right).
\]

Thus, it follows that
\begin{equation}\label{eq:estnum1}
\left|y_{HM}(x)-\frac{e^{-\frac{2}{3}x^{3/2}}}{2\sqrt{\pi}x^{1/4}}h_{b}\left(\frac{2}{3}x{}^{3/2}\right)\right|
\leqslant\frac{e^{-\frac{2}{3}x^{3/2}}}{2\sqrt{\pi}x^{1/4}}\left|\delta_{2}\left(\frac{2}{3}x{}^{3/2}\right)\right|
\leqslant\frac{9e^{-\frac{2}{3}x^{3/2}}}{640\sqrt{\pi}x^{13/4}},
\end{equation}
and
\begin{align}\label{eq:estnum2}
&\left|y_{HM}'(x)-\left(\frac{e^{-\frac{2}{3}x^{3/2}}}{2\sqrt{\pi}x^{1/4}}h_{b}\left(\frac{2}{3}x{}^{3/2}\right)\right)'\right|
\nonumber \\
&\leqslant\frac{e^{-\frac{2}{3}x^{3/2}}\left(4|\delta_{2}\left(\frac{2}{3}x^{3/2}\right)|x^{3/2}+
|\delta_{2}\left(\frac{2}{3}x{}^{3/2}\right)|+4|\delta_{2}'\left(\frac{2}{3}x^{3/2}\right)|x^{3/2}\right)}{8\sqrt{\pi}x^{5/4}}
\nonumber \\
&\leqslant\frac{9e^{-\frac{2}{3}x^{3/2}}\left(188x^{3/2}+25\right)}{64000\sqrt{\pi}x^{17/4}}.
\end{align}

Plugging the expression (\ref{eq:hbdef}) and the value $x=3$ into (\ref{eq:estnum1}) and (\ref{eq:estnum2}), we get
\begin{multline*}
\left|y_{HM}(3)-\frac{4}{607}\right|
\\
\leqslant \left|y_{HM}(3)-\frac{e^{-\frac{2}{3}3^{3/2}}}{2\sqrt{\pi}3^{1/4}}h_{b}\left(\frac{2}{3}3^{3/2}\right)\right|
+\left|\frac{4}{607}-\frac{e^{-\frac{2}{3}3^{3/2}}}{2\sqrt{\pi}3^{1/4}}h_{b}\left(\frac{2}{3}3^{3/2}\right)\right|
<8\cdot10^{-6},
\end{multline*}
and
\begin{multline*}
\left|y_{HM}'(3)+\frac{64}{5375}\right|\leqslant
\left|y_{HM}'(3)-\left(\frac{e^{-\frac{2}{3}x^{3/2}}}{2\sqrt{\pi}x^{1/4}}h_{b}\left(\frac{2}{3}x{}^{3/2}\right)\right)'\bigg|_{x=3}\right|\\
+\left|\frac{64}{5375}+\left(\frac{e^{-\frac{2}{3}x^{3/2}}}{2\sqrt{\pi}x^{1/4}}h_{b}\left(\frac{2}{3}x{}^{3/2}\right)\right)'\bigg|_{x=3}\right|<24\cdot10^{-6},
\end{multline*}
as desired.
\end{proof}

\section{Analysis of $y_{HM}$ for $0\leqslant x\leqslant3$}\label{sec:origin}

When studying $y_{HM}$ near the origin, it is more convenient to
use equation (\ref{eq:p2}) directly, without any change of variable. Since the region is finite, it is convenient to construct a quasi-solution using a simple polynomial. To construct the polynomial, we first numerically solve the equation \eqref{eq:p2} using the initial conditions $y(3)=\frac{4}{607}$ and $y'(3)=-\frac{64}{5375}$, which are close to the true values for $y_{HM}$ according to Corollary \ref{yhm3}. Then we fit the numerical solution using a polynomial of degree 11 under the maximum norm (similar to Chebyshev polynomials), and approximate the numerical coefficients of the polynomial with rational numbers to ensure mathematical rigor. As a result, we have

\begin{multline*}
y_{a}(x)=\frac{t^{11}}{1929701}-\frac{t^{10}}{625758}-\frac{t^{9}}{192428}+\frac{t^{8}}{27779}-\frac{t^{7}}{23450}-\frac{13t^{6}}{44056}+\frac{90t^{5}}{64211}\\
-\frac{19t^{4}}{21788}-\frac{125t^{3}}{9667}+\frac{1535t^{2}}{28314}-\frac{2759t}{28279}+\frac{1413}{19685},
\end{multline*}
where $t=x-\frac{3}{2}$.

Plugging $y_{a}$ into (\ref{eq:p2}), we get a remainder
\[
R_{3}(x)=y_{a}''(x)-2y_{a}(x)^{3}-xy_{a}(x),
\]
which is also a polynomial. We first need to show that this remainder is small for $0\leqslant x\leqslant3$.
\begin{rem} \label{polyest}
{\it Estimating a real polynomial $P(x)$ on an interval $[a,b]$ rigorously and with good accuracy is elementary. Here we use the simple method in \cite{lpm,dub}. We choose a suitable partition of $[a,b]$, $$\Pi=\{x_{0},x_{1},...,x_{n-1},x_{n}\},$$
where $x_{0}=a,x_{n}=b$, then write $$x=(x_{i}+x_{i-1})/2+u$$ on each subinterval $[x_{i-1},x_{i}]$ for $i=1,...,n,$ and re-expand $P$ to obtain a polynomial in $u$. The polynomial in $u$ is estimated by taking the extremum of the cubic sub-polynomial and bounding the rest terms by the sum of their absolute values. To be precise, if
\[
P(x)=\sum_{k=0}^{n}c_{k}\left(x-\frac{x_{i}+x_{i-1}}{2}\right)^{k},
\]
then we have
\[
\left|P(\frac{x_{i}+x_{i-1}}{2}+u)-\sum_{k=0}^{3}c_{k}u^{k}\right|\leqslant\sum_{k=4}^{n}|c_{k}|\left|\frac{x_{i}-x_{i-1}}{2}\right|^{k}
,\qquad |u|\leqslant \frac{|x_{i}-x_{i-1}|}{2}.
\]

This technique can also be used to show $P_{1}(x)\leqslant P_{2}(x)$
as it is equivalent to showing $P_{1}(x)-P_{2}(x)\leqslant0$.}
\end{rem}
We choose the partition $\{0,1/4,3/5,6/5,9/5,12/5,14/5,3\}$. By Remark
\ref{polyest}, it is easy to show that
\begin{equation}\label{eq:r3bound}
|R_{3}(x)|<18\cdot10^{-6}.
\end{equation}
Plugging $$y_{HM}=y_{a}+\delta_{3}$$ into (\ref{eq:p2}), we get
\begin{equation}\label{eq:delta3}
\delta_{3}''(x)=\delta_{3}(x)\left(6y_{a}(x)^{2}+x\right)+6\delta_{3}(x)^{2}y_{a}(x)+2\delta_{3}(x)^{3}-R_{3}(x)
=:F_{1}(\delta_{3}(x),x).
\end{equation}
Equation (\ref{eq:delta3}), together with the initial conditions $\delta_{3}(3)=y_{HM}(3)-y_{a}(3),~\delta_{3}'(3)=y_{HM}'(3)-y_{a}'(3)$
guarantee that $\delta_{3}=y_{HM}-y_{a}$. By Corollary \ref{yhm3}
and direct calculation, we have
\begin{equation*}\label{eq:delta3d}
\left|\delta_{3}(3)\right|<85\cdot10^{-7},\qquad \left|\delta_{3}'(3)\right|<241\cdot10^{-7}.
\end{equation*}
It is possible to show that $\delta_{3}$ is small by constructing
approximate Green's functions and using the contraction mapping principle
as in \cite{lpm,dub}. However, here we will use a simpler method
relying only on elementary inequalities. The idea is to construct
an explicit function satisfying a ``stronger'' ODE, so that $\delta_{3}$
must be bounded by that function. This is essentially the same idea as in Gr\"{o}nwall's inequality, where we replace the exponential bound by a more general function in order to obtain more accurate estimates for nonlinear equations. To be precise, we have the following
lemma.
\begin{lem}[Global existence of ODE solutions] \label{lem:estbound}
Consider the equation $$u''(x)=F(u(x),x)$$
on the interval $[a,b]$ with initial conditions $$u(a)=\alpha, \qquad u'(a)=\tilde{\alpha},$$
where $F$ is Lipschitz continuous in $\overline{B(\alpha,\beta)}\times[a,b]$ (here $\overline{B(\alpha,\beta)}$ denotes
the closed ball centered at $\alpha$ with radius $\beta$ in the complex plane). Suppose there exists an integrable function $G$ such that $$G(t,x)\geqslant|F(t,x)|$$
in $\overline{B(\alpha,\beta)}\times[a,b]$ and $G(t,x)$ is increasing
in $|t|$ (i.e., if $|t_{1}|\geqslant|t_{2}|$, then $G(t_{1},x)\geqslant G(t_{2},x)$).
Furthermore, suppose there exists a function $y$ such that
$$y(a)>|\alpha|,\quad y'(a)\geqslant|\tilde{\alpha}|,\quad y''(x)\geqslant G(y(x),x).$$
Then there exists a unique solution $u(x)$ on $[a,b]$ such that
$$\left|u(x)\right|<y(x) \quad \textrm{and} \quad \left|u'(x)\right|\leqslant y'(x)$$
for all $x\in[a,b]$. Furthermore, if the conditions $y(a)>|u(a)|,\: y'(a)\geqslant \left|u'(a)\right|$
are replaced with $y(b)>|u(b)|,\: y'(b)\leqslant-\left|u'(b)\right|$, the conclusions
still hold.
\end{lem}
\begin{proof}
Local existence and uniqueness of a solution $u$ near $x=a$ follows
from Picard's Existence Theorem (cf. \cite{cod}). We define $$u_{1}(x)=|\alpha|+\left|\tilde{\alpha}\right|(x-a)+\int_{a}^{x}\int_{a}^{t}G(u(s),s)dsdt.$$
By straightforward calculation, we have
\begin{multline*}
\left|u(x)\right|=\left|\alpha+\tilde{\alpha}(x-a)+\int_{a}^{x}\int_{a}^{t}F(u(s),s)ds dt\right|
\\
\leqslant|\alpha|+\left|\tilde{\alpha}\right|(x-a)+\int_{a}^{x}\int_{a}^{t}|F(u(s),s)|dsdt\leqslant u_{1}(x).
\end{multline*}

Let $$y_{1}(x)=y(a)+y'(a)(x-a)+\int_{a}^{x}\int_{a}^{t}G(y(s),s)dsdt.$$
By direct calculation, we see that
$$y_{1}(a)=y(a),\quad y_{1}'(a)=y'(a),\quad y_{1}''(x)\leqslant y''(x),$$
which means $y_{1}(x)\leqslant y(x)$. Let $b_{1}=\sup_{c\in[a,b]}\{y(x)>|u(x)|,x\in[a,c]\}$.
For $a\leqslant x\leqslant b_{1}$, we have $$\int_{a}^{x}\int_{a}^{t}G(u(s),s)dsdt\leqslant\int_{a}^{x}\int_{a}^{t}G(y(s),s)dsdt,$$
because $G(t,x)$ is increasing in $|t|$. Since $y(a)>|u(a)|$, using
definitions of $u_{1}$ and $y_{1}$, we see that $$u_{1}(x)<y_{1}(x)$$
for all $a\leqslant x\leqslant b_{1}$. Now suppose $b_{1}<b$. For every $x<b_{1}$, Picard's Existence Theorem implies that the solution $u$ exists in $[x,\min(x+\beta/\sup|F|,b)]$. Thus the solution $u$ exists for all $x<\min(b_{1}+\beta/\sup|F|,b)$. Then by continuity,
we would have $u_{1}(x)<y_{1}(x)$ for $x\in[a,b_{1}+\delta]$ and for some $\delta>0$, which would contradict with the definition of $b_{1}$. Hence, $$|u(x)|\leqslant u_{1}(x)<y_{1}(x)\leqslant y(x)$$ for all $x\in[a,b]$.

To estimate $u'$, note that
\begin{multline*}
\left|u'(x)\right|=\left|\tilde{\alpha}+\int_{a}^{x}F(u(s),s)ds\right|\leqslant\left|\tilde{\alpha}\right|+\int_{a}^{x}G(u(s),s)ds
\\
\leqslant y'(a)+\int_{a}^{x}G(y(s),s)ds \leqslant y'(x),
\end{multline*}
where we integrated the known inequality $y''(x)\geqslant G(y(x),x)$
in the last step.

Finally, if conditions $y(a)>|u(a)|,\: y'(a)\geqslant\left|u'(a)\right|$ are
replaced with $y(b)>|u(b)|,\: y'(b)\leqslant-\left|u'(b)\right|$, we can simply
make the change of variable $x\rightarrow-x$ and apply the previous result.
\end{proof}

Now we are ready to prove the following result:
\begin{prop}\label{prop:yhm0}
We have the estimates
$$\left|y_{HM}(0)-\frac{98}{267}\right|<11\cdot10^{-4},
\qquad \left|y_{HM}'(0)+\frac{153}{518}\right|<12\cdot10^{-4}.$$
\end{prop}
\begin{proof}
Clearly $F_{1}$ in (\ref{eq:delta3}) is Lipschitz continuous in
every bounded region. Let $$G(t,x)=|t|\left(6y_{a}(x)^{2}+x\right)+6y_{a}(x)|t|{}^{2}+2|t|^{3}+18\cdot10^{-6}.$$
Since $y_{a}$ is a polynomial, by using Remark \ref{polyest} with
partition $\{0,1,2,3\}$, we easily see that $y_{a}$ is positive on
$[0,3]$. Therefore $G(t,x)$ is increasing in $|t|$. By (\ref{eq:r3bound}),
we see that $$G(t,x)\geqslant\left|F_{1}(t,x)\right|,\qquad 0\leqslant x\leqslant3.$$
Now we set
\begin{multline*}
y_{1}(x)=\frac{s^{8}}{55140149}-\frac{s^{7}}{15591646}-\frac{s^{6}}{35492113}+\frac{s^{5}}{526490476}+\frac{s^{4}}{144870}\\
-\frac{3s^{3}}{63886}+\frac{7s^{2}}{46477}-\frac{43s}{172565}+\frac{5}{29696},
\end{multline*}
where $s=x-\frac{3}{2}$.

By direct calculation, we see
$$y_{1}(3)>9\cdot10^{-7}>|\delta_{3}(3)|\quad \textrm{and} \quad y_{1}'(3)<-245\cdot10^{-6}<-\left|\delta_{3}'(3)\right|.$$ Since $y_{1}''(x)-G(y_{1}(x),x)$
is a polynomial, using Remark \ref{polyest} with partition $\{0,1,2,3\}$,
we see that $y_{1}''(x)-G(y_{1}(x),x)\geqslant0$. Therefore,
by Lemma \ref{lem:estbound},
we have $$|\delta_{3}(x)|\leqslant y_{1}(x)\quad \textrm{and} \quad \left|\delta_{3}'(x)\right|\leqslant y_{1}'(x).$$
In particular, one has
\[
\left|y_{HM}(0)-\frac{98}{267}\right|\leqslant\left|y_{a}(0)-\frac{98}{267}\right|+|\delta_{3}(0)|\leqslant\left|y_{a}(0)-\frac{98}{267}\right|+y_{1}(0)<11\cdot10^{-4},
\]
and
\[
\left|y_{HM}'(0)+\frac{153}{518}\right|\leqslant\left|y_{a}'(0)+\frac{153}{518}\right|+|\delta_{3}'(0)|\leqslant\left|y_{a}'(0)+\frac{153}{518}\right|+|y_{1}'(0)|<12\cdot10^{-4}.
\]
\end{proof}

\section{Analysis of $y_{HM}$ in the finite domain $\Omega_{2}$}\label{sec:left0}

In this section we will show that $y_{HM}$ is pole-free in $\Omega_{2}$, where $\Omega_{2}$
is defined in \eqref{def:omega2}. This result, combined with Proposition \ref{prop:farleft}, will be sufficient to
prove Theorem \ref{main}. Our strategy is still to construct a quasi-solution and to use Lemma \ref{lem:estbound} to bound the error. However, since $\Omega_{2}$ is a sector in the complex plane, we first make a remark about
estimating complex rational functions.
\begin{rem}\label{polycest}
{\it Assume $\Omega_{p}$ is a closed polygonal domain
in the complex plane and $$F(z)=\frac{P(z)}{Q(z)},$$
where $P,Q$ are complex polynomials and $Q$ has no zero in $\Omega_{p}$. To estimate the
modulus of $F$, we note that since $F$ is analytic in $\Omega_{p}$,
by the maximum modulus principle, it is sufficient to estimate its
modulus along the boundary $\partial\Omega_{p}$. Since $\Omega_{p}$
is a polygon, its boundary consists of line segments. On each line
segment $[z_1,z_2]$, we have $z=z_{1}+z_{2}t,~0\leqslant t\leqslant1$, and we note
that
\begin{equation}\label{eq:complexineq}
\left|\frac{P(z_{1}+z_{2}t)}{Q(z_{1}+z_{2}t)}\right|\leqslant M\Leftrightarrow|P(z_{1}+z_{2}t)|^{2}\leqslant M^{2}|Q(z_{1}+z_{2}t)|^{2}.
\end{equation}

Since $|P(z_{1}+z_{2}t)|^{2}$ and $|Q(z_{1}+z_{2}t)|^{2}$ are both
real polynomials in $t$, (\ref{eq:complexineq}) could be proved
using the method in Remark \ref{polyest}.}
\end{rem}

In view of Remark \ref{polycest}, we consider the right triangular domain $\tilde{\Omega}_{2}$ with vertices at $0,-\frac{9}{2\sqrt{3}},-\frac{9}{4\sqrt{3}}+\frac{9}{4}i$. It
is easily seen that $\Omega_{2}\subseteq\tilde{\Omega}_{2}$, as illustrated in figure \ref{left00}.

\begin{figure}
\centering
\begin{overpic}[scale=0.8]{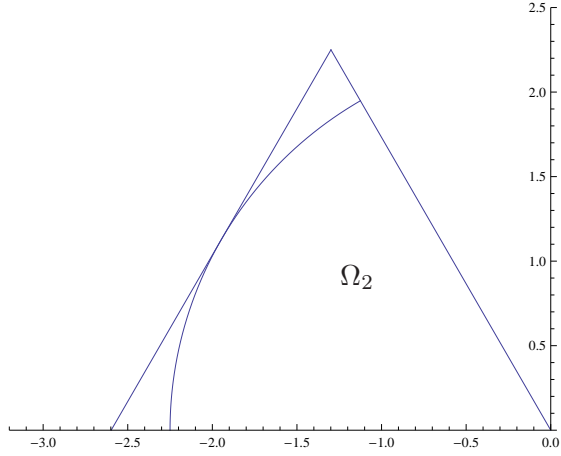}
\put(60,30){$\Omega_2$}
\end{overpic}
\caption{Sector $\Omega_{2}$ and triangular region $\tilde{\Omega}_{2}$.}
\label{left00}
\end{figure}

In order to find a quasi-solution in the complex region $\tilde{\Omega}_{2}$, we solve \eqref{eq:p2} numerically using initial conditions given in Proposition \ref{prop:yhm0} and use least-squares polynomial approximations, which results in the quasi-solution
\begin{multline}\label{eq:ybdef}
y_{b}(x)=\frac{x^{15}}{13206825}+\frac{x^{14}}{717099}+\frac{x^{13}}{81755}+\frac{x^{12}}{15201}+\frac{11x^{11}}{47200}+\frac{13x^{10}}{24088}+\frac{39x^{9}}{53333}\\
+\frac{18x^{8}}{61523}-\frac{17x^{7}}{20578}-\frac{93x^{6}}{35396}-\frac{224x^{5}}{30615}-\frac{360x^{4}}{36911}+\frac{203x^{3}}{10806}+\frac{33530x^{2}}{688889}-\frac{153x}{518}+\frac{98}{267},
\end{multline}

As usual, we need to estimate the remainder
\[
R_{4}(x)=y_{b}''(x)-2y_{b}(x)^{3}-xy_{b}(x).
\]
We note that by symmetry every polynomial or rational function $F(z)$
with real coefficients only needs to be estimated on two edges $$\Sigma_{1}:z=\frac{9}{4}\left(-\frac{2}{\sqrt{3}}+\left(\frac{1}{\sqrt{3}}+i\right)t\right) \quad
\textrm{and} \quad \Sigma_{2}:z=\frac{9}{4}\left(-\frac{1}{\sqrt{3}}+i\right)t,$$
where $0\leqslant t\leqslant1$.
This is because by the maximum modulus principle $F$ only need to
be estimated on the boundary of $\tilde{\Omega}_{2}\cup\tilde{\Omega}_{2}^{c}$,
where $\tilde{\Omega}_{2}^{c}$ is complex conjugate of $\tilde{\Omega}_{2}$,
namely the reflection of $\tilde{\Omega}_{2}$ with respect to the
real axis. However, since $|F(z)|=|F(\bar{z})|$, it is sufficient
to estimate $F$ on the two edges in the upper half plane.

We take the partition $\{0,2/5,4/5,14/15,1\}$ for the edge $\Sigma_{1}$,
and the partition \newline
$\{0,1/3,3/4,14/15,1\}$ for the edge $\Sigma_{2}$.
Using Remark \ref{polycest}, we see that
\[
|R_{4}(x)|^{2}<3\cdot10^{-5}\textrm{ on }\Gamma_{1,2}\Rightarrow|R_{4}(x)|<\frac{3}{500}\textrm{ on }\tilde{\Omega}_{2}.
\]

Now we are ready to show the following
\begin{prop}\label{prop:left0}
The Hastings-solution $y_{HM}$
satisfies the bound $$|y_{HM}(x)-y_{b}(x)|<6/5$$
in $\Omega_{2}$. In particular, $y_{HM}$ is pole-free in $\Omega_{2}$.
\end{prop}
\begin{proof}
We substitute $y_{HM}=y_{b}+\delta_{4}$ into (\ref{eq:p2}) and get
\begin{equation}
\delta_{4}''(x)=\left(6y_{b}(x)^{2}+x\right)\delta_{4}(x)+6y_{b}(x)\delta_{4}(x)^{2}+2\delta_{4}(x)^{3}-R_{4}(x)=:F_{2}(\delta_{4}(x),x)\label{eq:delta4}
\end{equation}

From (\ref{eq:ybdef}), it is clear that $y_{b}(0)=\frac{98}{267}$ and
$y_{b}'(0)=-\frac{153}{518}$. This, together with Proposition \ref{prop:yhm0}, implies that
\begin{equation}\label{eq:delta40}
|\delta_{4}(0)|<11\cdot10^{-4},\qquad \left|\delta_{4}'(0)\right|<12\cdot10^{-4}.
\end{equation}

Now we will use Lemma \ref{lem:estbound} to estimate $\delta_{4}(x)$
along radial paths $x=re^{i\theta}\:(0\leqslant r\leqslant\frac{9}{4})$
for all fixed $\frac{2\pi}{3}\leqslant\theta\leqslant\pi$. With the
change of variable $\tilde{\delta}_{4}(r)=\delta_{4}(re^{i\theta})$,
equation (\ref{eq:delta4}) becomes
\begin{equation}\label{eq:f2tt}
\tilde{\delta}_{4}''(r)=e^{2i\theta}F_{2}(\tilde{\delta}_{4}(r),re^{i\theta})=:\tilde{F}_{2}(\tilde{\delta}_{4}(r),r).
\end{equation}

To find a suitable $G$ required by Lemma \ref{lem:estbound}, we
first estimate relevant terms in $\tilde{F}_{2}$. We take the same
partition $\{0,1/2,14/15,1\}$ for both $\Sigma_{1}$ and $\Sigma_{2}$.
Using Remark \ref{polycest} we see that
\begin{equation}
|y_{b}(x)|^{2}<\frac{83}{50}\textrm{ on }\Sigma_{1,2}\Rightarrow6|y_{b}(x)|<8\textrm{ on }\tilde{\Omega}_{2}\label{eq:yb2}
\end{equation}

Similarly, we take the same partition $\{0,1/2,3/4,1\}$ for both
$\Sigma_{1}$ and $\Sigma_{2}$ and use Remark \ref{polycest}, which
gives
\begin{multline}\label{eq:yb1}
\left|6y_{b}(x)^{2}+x\right|^{2}-\frac{144}{25}|x-1|^{2}<0\textrm{ on }\Sigma_{1,2} \\
\Rightarrow\left|\frac{6y_{b}(x)^{2}+x}{x-1}\right|<\frac{12}{5}\textrm{ on }\tilde{\Omega}_{2}\Rightarrow|6y_{b}(x)^{2}+x|<\frac{12}{5}|x-1|\textrm{ on }\tilde{\Omega}_{2}.
\end{multline}

Now we define
\[
G_{2}(t,r)=\frac{12}{5}(r+1)|t|+8|t|^{2}+2|t|^{3}+\frac{3}{500}.
\]
Obviously $G_{2}(t,r)$ is increasing in $|t|$. By (\ref{eq:delta4}),
(\ref{eq:f2tt}), (\ref{eq:yb2}) and (\ref{eq:yb1}), it is clear
that $G_{2}(t,r)\geqslant|\tilde{F}_{2}(t,r)|$. Let
\begin{multline}
y_{2}(r)=\frac{400t^{10}}{11977}+\frac{265t^{9}}{11857}-\frac{855t^{8}}{10951}-\frac{149t^{7}}{6608}
\\
+\frac{293t^{6}}{2551}+\frac{1013t^{5}}{14669}
+\frac{128t^{4}}{6441}+\frac{424t^{3}}{6079}+\frac{1261t^{2}}{13159}+\frac{549t}{7508}+\frac{267}{9871},
\end{multline}
where $t=r-1$.

By direct calculation and (\ref{eq:delta40},) we have $$y_{2}(0)>11\cdot10^{-4}>\left|\tilde{\delta}_{4}(0)\right| \quad \textrm{
and} \quad y_{2}'(0)>12\cdot10^{-4}>\left|\tilde{\delta}_{4}'(0)\right|.$$

Since $y_{2}''(r)-G_{2}(y_{2}(r),r)$ is a polynomial, we take the partition
$\{0,1/5,1/2,1,3/2,9/5,2,9/4\}$ and use Remark \ref{polyest} to
get
\[
y_{2}''(r)-G_{2}(y_{2}(r),r)>0 \textrm{ for } r\in[0,9/4].
\]

By Lemma \ref{lem:estbound}, we have $\left|\delta_{4}(r)\right|=\left|\tilde{\delta}_{4}(r)\right|\leqslant y_{2}(r)$
for $r\in[0,9/4]$. By taking the partition $\{0,1,9/4\}$ and using
Remark \ref{polyest}, we see that $y_{2}(r)<6/5$. Thus $|y_{HM}-y_{b}|<6/5$.
\end{proof}

\section{Proofs of Theorems \ref{main} and \ref{main-1}}\label{sec:proofs}
\begin{proof}[Proof of Theorem \ref{main}]
The theorem simply follows from Proposition \ref{prop:farleft}
and Proposition \ref{prop:left0}. Note that $\Omega\subseteq\Omega_{0}\cup\Omega_{2}$,
since $\frac{3^{4/3}}{2}<\frac{9}{4}$.\end{proof}

\begin{proof}[Proof of Theorem \ref{main-1}]
By the argument at the beginning of Section \ref{sec:about} and Theorem \ref{main}, it follows that
$y_{HM}$ is pole-free in the sector $\arg x \in [\frac{2\pi}{3},\frac{ 4 \pi}{3}]$. For the region $\arg x\in [-\pi/3,\pi/3]$, it is already known that $y_{HM}$ is pole-free (see \cite{bert}).
\end{proof}

\begin{rem} Our method also applies in the region $\arg x\in [-\pi/3,\pi/3]$. In fact, we only need to change the definition of $\mathcal{S}_{3}$ to include the sector $[-\pi/3,\pi/3]$ in complex plane in Proposition \ref{prop:delta2}, and to construct a quasi-solution similar to $y_b$ in \eqref{eq:ybdef} for $\{x\in\mathbb{C}:|x|\leqslant 3,-\pi/3 \leqslant\arg x\leqslant\pi/3\}$. The details are left to the interested readers.

Furthermore, we would like to mention that one can also adopt the method to
prove other cases of Conjecture \ref{conj}. However, since there are infinitely many $2$-truncated solutions of PI and PII with significantly different asymptotic behaviors, it is not clear to us at the moment how to find a uniform proof that works for the general case.

\end{rem}


\section*{Acknowledgements}
We would like to express our sincere gratitude and appreciation to Roderick Wong and Dan Dai for organizing a series of conferences and seminars, which facilitated our collaboration on this project, as well as other related discussions.

Min Huang was supported by the Early Career Scheme 21300114 of Research Grants Council of Hong Kong. Shuai-Xia Xu was supported in part by the National Natural Science Foundation of China under grant number 11201493, GuangDong Natural Science Foundation under grant number S2012040007824, Postdoctoral Science Foundation of China under Grant No.2012M521638, and the Fundamental Research Funds for the Central Universities under grant number 13lgpy41. Lun Zhang was partially supported by The Program for Professor of Special Appointment (Eastern Scholar) at Shanghai Institutions of Higher Learning (No. SHH1411007) and by Grant SGST 12DZ 2272800, EZH1411513 from Fudan University.

\end{document}